\documentclass[12pt]{amsart}

\usepackage{fullpage}

\usepackage{amsmath, amssymb}
\usepackage{graphics}
\usepackage{mathtools}
\usepackage{xcolor}
\usepackage{enumerate, marvosym}
\usepackage{url}

\usepackage[parfill]{parskip}

\newtheorem{thm}{Theorem}[section]
\newtheorem{prop}[thm]{Proposition}

\newtheorem{lemma}[thm]{Lemma}
\newtheorem{cor}[thm]{Corollary}

\theoremstyle{definition}
\newtheorem{definition}[thm]{Definition}

\newtheorem{remark}[thm]{Remark}

\newcommand{\der}{\partial}
\newcommand{\bra}[1]{\left( #1 \right)}
\newcommand{\set}[1]{\left \{#1 \right \}}
\newcommand{\bs}{\backslash}
\newcommand{\ssq}{\subseteq}
\newcommand{\allstar}[1]{\[\begin{aligned} #1  \end{aligned}\]}
\newcommand{\gen}[1]{\langle #1 \rangle}
\newcommand{\bbm}{\begin{bmatrix}}
\newcommand{\ebm}{\end{bmatrix}}
\newcommand*{\longhookrightarrow}{\ensuremath{\lhook\joinrel\relbar\joinrel\rightarrow}}
\newcommand*{\inj}{\longhookrightarrow}
\newcommand{\wtilde}{\widetilde}
\DeclareMathOperator{\height}{ht}
\newcommand{\ra}{\rightarrow}
\newcommand{\mbc}{\mathbb{C}}

\newcommand{\piecewise}[3]{\begin{cases}
#1 & \text{if }#2 \\
#3 & \text{otherwise }
\end{cases}}
\newcommand{\midskip}{\hspace{0.1in}}
\newcommand{\piecewiseThree}[5]{\begin{cases}
#1 & \text{if }#2 \\
#3 & \text{if }#4 \\
#5 & \text{otherwise }
\end{cases}}

\newcommand{\sjoin}{\text{\Lightning}}
\newcommand{\bolt}{\sjoin}
\newcommand{\paral}{\star}

\renewcommand{\P}{\mathcal{P}}

\newcommand{\kirk}[1]{{#1}}%{{\Psi_{#1}}}
\newcommand{\breaker}[1]{{\overline{#1}}}%{{P_{#1}}}
\newcommand{\dkirk}[2]{\kirk{#1}^{#2}}%{\kirk{#1}_{#2}}
\newcommand{\dbreaker}[2]{\breaker{#1}^{#2}}%{\breaker{#1}_{#2}}

\title{Some results on an algebro-geometric condition on graphs}

\author{Avi Kulkarni, Gregory Maxedon, and Karen Yeats}

\thanks{AK was supported by an NSERC PGS D. GM was supported by an NSERC USRA during this project.  KY is supported by an NSERC discovery grant.}

\begin{document}

\newcommand{\Ais}{\overset{12}{A}}
\newcommand{\Aisef}{\overset{12,34}{A}}
\newcommand{\Aiseg}{\overset{12,35}{A}}
\newcommand{\Aisfg}{\overset{12,45}{A}}
\newcommand{\Aisefg}{\overset{12,345}{A}}
\newcommand{\Aef}{\overset{34}{A}}
\newcommand{\Aeg}{\overset{35}{A}}
\newcommand{\Afg}{\overset{45}{A}}
\newcommand{\Aefg}{\overset{345}{A}}
\newcommand{\Aa}{A^{a_i}} 
\newcommand{\Aisa}{\overset{12}{A}\overset{{a_i}}{\textcolor{white}{a}}}
\newcommand{\Aefa}{\overset{34}{A}\overset{{a_i}}{\textcolor{white}{a}}}
\newcommand{\Aega}{\overset{35}{A}\overset{{a_i}}{\textcolor{white}{a}}}
\newcommand{\Afga}{\overset{45}{A}\overset{{a_i}}{\textcolor{white}{a}}}
\newcommand{\Aefga}{\overset{345}{A}\overset{{a_i}}{\textcolor{white}{a}}}
\newcommand{\Bis}{\overset{12}{B}}
\newcommand{\Bie}{\overset{13}{B}}
\newcommand{\Bieb}{{\overset{13}{B}\overset{b_j}{\textcolor{white}{ _*}}}}
\newcommand{\Bef}{\overset{34}{B}}
\newcommand{\Bb}{B^{b_j}}

\begin{abstract}
Paolo Aluffi, inspired by an algebro-geometric problem, asked when the Kirchhoff polynomial of a graph is in the Jacobian ideal of the Kirchhoff polynomial of the same graph with one edge deleted.  

We give some results on which graph-edge pairs have this property.  In particular we show that multiple edges can be reduced to double edges, we characterize which edges of wheel graphs satisfy the property, we consider a stronger condition which guarantees the property for any parallel join, and we find a class of series-parallel graphs with the property.
\end{abstract}

\maketitle

\section{Introduction}

Over the last decade there has been an interest in taking an algebraic geometry inspired approach to understanding Feynman integrals \cite{Bljapan, bek, Brbig, BrS, BrS3, BrSY, BrY, Doreg, Mmotives}.  The key object of study is the \emph{graph hypersurface} which we can define as follows.  Given a multigraph $G$ (henceforth we will just say graph with the understanding that multiple edges and self-loops are permitted) to each edge $e$ of $G$ assign a variable $t_e$ and define the \emph{Kirchhoff polynomial}\footnote{This polynomial is also known as the first Symanzik polynomial, and sometimes the Kirchhoff polynomial is instead defined dually with the condition $e\in T$ in place of $e\not\in T$, see for example \cite{BW}.} of $G$ by
\[
\Psi_G = \sum_{\substack{T \text{ spanning}\\\text{tree of }G}}\prod_{e \not\in T}t_e.
\]
The graph hypersurface is then simply the variety given by the zero set of this polynomial, viewed in projective space or in affine space depending on context.  This relates back to Feynman integrals because the Kirchhoff polynomial plays a key role in the integrand of the Feynman integral of $G$ in parametric form, see for example \cite{BW}.  In fact, viewing $G$ as a massless scalar Feynman diagram, the first interesting piece of the Feynman integral is simply
\[
 \int_{t_i \geq 0} \frac{\Omega}{\Psi_G^2}
\]
where $\Omega = \sum_{i=1}^{|E(G)|}(-1)^i dt_1 \cdots dt_{i-1} dt_{i+1} \cdots dt_{|E(G)|}$.
This has come to be known as the Feynman period, see \cite{Sphi4}, and is a very interesting object physically, number theoretically, and combinatorially.

\medskip
The interplay between Kirchhoff polynomials for different graphs will be crucial to our argument.  Therefore to keep the notation light we will abuse notation and simply write $\kirk{G}$ for~$\Psi_G$.  We will take the convention that the Kirchhoff polynomial of a disconnected graph is 0.

At the level of the Kirchhoff polynomial, edge deletion, equivalently partial derivative, will be indicated by superscripts and edge contraction, equivalently setting variables to $0$, by subscripts.  That is, for Kirchhoff polynomials in this notation
\[
\kirk{G}^e = \kirk{G\smallsetminus e} = \partial_e \kirk{G} \quad \text{and} \quad \kirk{G}_e = \kirk{G/e} = \kirk{G}|_{e=0}.
\]
These facts are elementary consequences of the definition of the Kirchhoff polynomial.  A similar argument gives the classical contraction deletion relation
\[
G = t_eG^e + G_e \quad \text{for $e$ not a bridge or self-loop}.
\]

\medskip

Aluffi and Marcolli in \cite{AMfeyn} gave a definition of algebro-geometric Feynman rules which captured the most basic properties of Feynman rules in quantum field theory.  Specifically, they require the multiplicative property for disjoint unions of graphs, which is a restatement of the multiplicative property of independent events in basic probability; and they require the formula for recasting Feynman diagrams as trees of one-particle-irreducible diagrams.  

Other more advanced properties of physical Feynman rules can also be captured in algebraic language, see for example \cite{Pmsc}, but that is another story.

Aluffi and Marcolli then look at examples of their algebro-geometric Feynman rules which appear natively in the land of algebraic geometry.  One example comes from classes in the Grothendieck ring; another comes from Chern class calculations and gives univariate polynomial output.

In the course of studying when this second example satisfies contraction-deletion relations, Aluffi in \cite{Acond} needed to assume two technical conditions, which he calls Condition 1 and Condition 2, each of which is a condition on a pair of a graph $G$ along with an edge $e$ of~$G$.

\medskip

Aluffi's condition 1 is not difficult to state. Throughout the paper we follow \cite{Acond} and work over $\mathbb{Q}$ (See \cite[Section 2.4]{Acond}).
\begin{definition}
Let $G$ be a connected graph and $e$ an edge of $G$.  Then condition 1 for the edge $e$ of $G$, written $1(G,e)$, is the statement
\[
\kirk{G} \in \langle \partial \kirk{G}^e \rangle
\]
where $\langle \partial \kirk{G}^e \rangle$ is the ideal of partial derivatives, i.e. the Jacobian ideal, of $\kirk{G}^e$, the Kirchhoff polynomial of $G$ with $e$ deleted.
\end{definition}

There are a few immediate observations worth making.  From the contraction deletion relation and Euler's homogeneous function theorem, see Theorem~\ref{thm euler}, condition 1 for regular edges is equivalent to the statement
\[
\kirk{G}_e \in \langle \partial \kirk{G}^e \rangle.
\]

Certain special cases of edges are easy to understand.  If $e$ is a self-loop then $\kirk{G} = t_e\kirk{G}^e$ and so $1(G,e)$ is true.  If $e$ is a bridge then $\kirk{G}^e = 0$ since $G\smallsetminus e$ is disconnected and so has no spanning trees; thus $1(G,e)$ is false.  Finally, if $G\smallsetminus e$ is a tree then $\kirk{G}^e = 1$ and so again $1(G,e)$ is false.  Aluffi calls an edge which does not fall into one of the previous cases \emph{regular}.

\medskip

In this paper we investigate the graph theoretic underpinning of condition 1.  We are not able to obtain a full characterization of graph-edge pairs which satisfy the condition, but we do obtain the following interesting results.  Propositions \ref{2 imp 1} and 
\ref{2 imp 3} give that multiple edges of any multiplicity greater than 1 are equivalent to double edges from the perspective of condition 1.  Propositions \ref{rim prop} and \ref{spoke prop} show that for wheels with at least 4 spokes, condition 1 is false for all rim edges and true for all spoke edges.  Then we move to focusing on series-parallel graphs.  Definition~\ref{simultaneous combination} gives a stronger condition which, by Corollary~\ref{stability cor}, shows when condition 1 is stable under parallel join.  Finally, Corollary~\ref{S-for-co-Hamiltonians} describes a class of series-parallel graphs where condition 1 holds for all edges and Proposition~\ref{replacement by co-Hamiltonian} builds from this a much larger class of series-parallel graphs (and some other graphs) with specific edges for which condition 1 holds.

\section{Preliminaries}
For our arguments we want to consider cases of identifying vertices.  We will use the following notation
\[
\overset{12}{G}
\]
for the Kirchhoff polynomial of the graph $G$ with vertices $1$ and $2$ identified, and more generally if $s_1, s_2, \ldots$ are sets of vertices of $G$ then
\[
\overset{s_1, s_2, \ldots}{G}
\]
is the Kirchhoff polynomial of the graph $G$ with the vertices of $s_1$ identified, the vertices of $s_2$ identified and so on.

By considering the possible spanning trees we can write down the 1- and 2-cut formulas for the Kirchhoff polynomial.  Specifically, if $G$ is formed from $H$ and $H'$ joined at a vertex then
\[
G = HH'
\]
and if $G$ is formed from $H$ and $H'$ joined at two vertices, 1 and 2, then
\[
G = \overset{12}{H}H' + H\overset{12}{H'}.
\]

One of the main algebraic tools we'll use is Euler's homogeneous function theorem.
 \begin{thm}[Euler's Theorem]\label{thm euler}
Let $R$ be a ring and  $f \in R[x_1, \ldots, x_n]$ be a homogeneous function of degree $m$. Then
\[
m \cdot f = \sum_{j = 1}^n x_j \frac{\der f}{\der x_j}.
\]
\end{thm}

Next, we will give a few propositions explaining how small vertex and edge cuts affect condition 1.
%Propositions 
%\ref{1vertex} and 
%\ref{2edge} 
%are important because they imply that condition 1 only needs to be checked on 2-connected graphs with no two edge cut sets.  The condition on the edges of all other graphs can be obtained from these graphs by these propositions.  

\begin{prop}\label{1vertex}
If $ G $ and $ H $ are connected graphs that are joined at one vertex and $ e $ is a regular edge of $ G $, then $ 1(G,e) \Leftrightarrow 1(G \cup H, e) $.

% % % % % % % % % % %
% % D R A W I N G % %
% % % % % % % % % % %
\begin{center}
\includegraphics[width=0.36\textwidth]{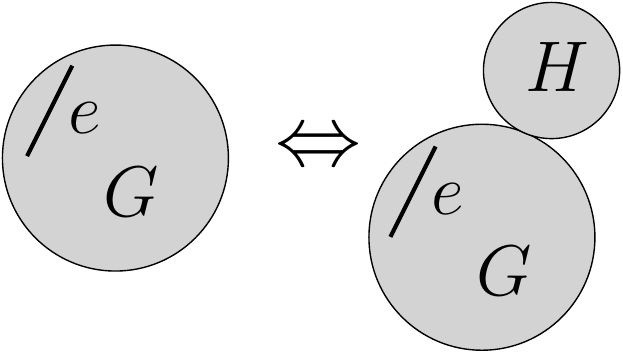}
\end{center}

\end{prop}

\begin{proof}
Let $ G $ and $ H $ be as in the statement.  

$ (\Rightarrow) $ Suppose $ 1(G,e) $, so for some $ P_i $ in the edge variables $ a_i $ of $ G $,
$$ G_e = \sum_i P_i G^{ e a_i } .$$
Then, since neither $ e $ nor the $ a_i $ are variables of $ H $,
$$ 
\begin{aligned} 
(G \cup H)_e 
= G_e \cup H 
= H G_e 
&= H \sum_i P_i G^{e a_i} \\
&= \sum_i P_i H G^{e a_i} \\
&= \sum_i P_i (H \cup G^{e a_i}) \\
&= \sum_i P_i (G \cup H)^{e a_i}. \\
\end{aligned}
$$
Hence $ 1(G \cup H, e) $ .

$ (\Leftarrow) $ Suppose $ 1(G \cup H, e) $, so for some sets of polynomials $ \{P_i\} $ and $ \{Q_j\} $, with all polynomials in the edge variables $ a_i $ of $ G $ and $ b_j $ of $ H $,
$$ 
\begin{aligned}
(G \cup H)_e = H G_e 
&= \sum_i P_i (G \cup H)^{e a_i} + \sum_j Q (G \cup H)^{e b_i} \\
&= \sum_i P_i H G^{e a_i} + \sum_j Q_j G^e H^{b_j} \\
&= H \sum_i P_i G^{e a_i} + G^e \sum_j Q_j H^{b_j}
. \\
\end{aligned}
$$
Set all of the $ b_j $ equal to $ 1 $.  Then
$$ 
\begin{aligned}
H|_{\substack{b_j = 1 \\ b_j \text{ edge} \\\text{of }H}}G_e
&= H|_{\substack{b_j = 1 \\ b_j \text{ edge} \\\text{of }H}} \sum_i P_i|_{\substack{b_j = 1 \\ b_j \text{ edge}\\\text{of }H}} G^{e a_i} + G^e \left.\left(\sum_u Q_u H^{b_u}\right)\right|_{\substack{b_j = 1 \\ b_j \text{ edge}\\\text{of }H}}
. \\
\end{aligned}
$$
So
$$ G_e = \sum_i R_i G^{e a_i} + S G^e, $$
where the $ R_i $ and $ S $ are polynomials in the variables $ a_i $.  Since Theorem 
\ref{thm euler} 
implies  $ G^e \in \langle \partial G^e \rangle $ and the $ G^{e a_i} $ are themselves partials, we have $ 1(G,e) $.
\end{proof}

\begin{prop}\label{2edge}
Whenever there is a two edge cut-set, contracting one of these edges has no effect on condition 1 for the remaining edges. Formally, let $G$ be a graph with a two edge-cut set $\set{x,y}$. Then for all $e \in G \bs \set{x}$ we have $1(G,e) \iff 1(G_x,e)$.

% % % % % % % % % % %
% % D R A W I N G % %
% % % % % % % % % % %
\begin{center}
\includegraphics[width=0.36\textwidth]{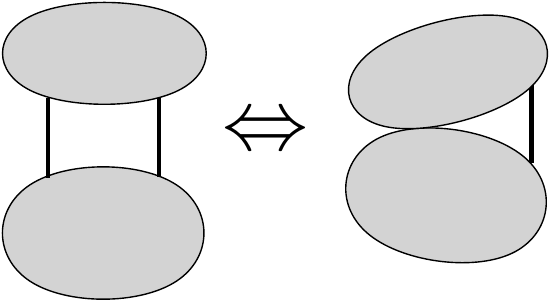}
\end{center}

\end{prop}

\begin{proof}
We can draw $G$ as

% % % % % % % % % % %
% % D R A W I N G % %
% % % % % % % % % % %
\begin{center}
\includegraphics[width=0.40\textwidth]{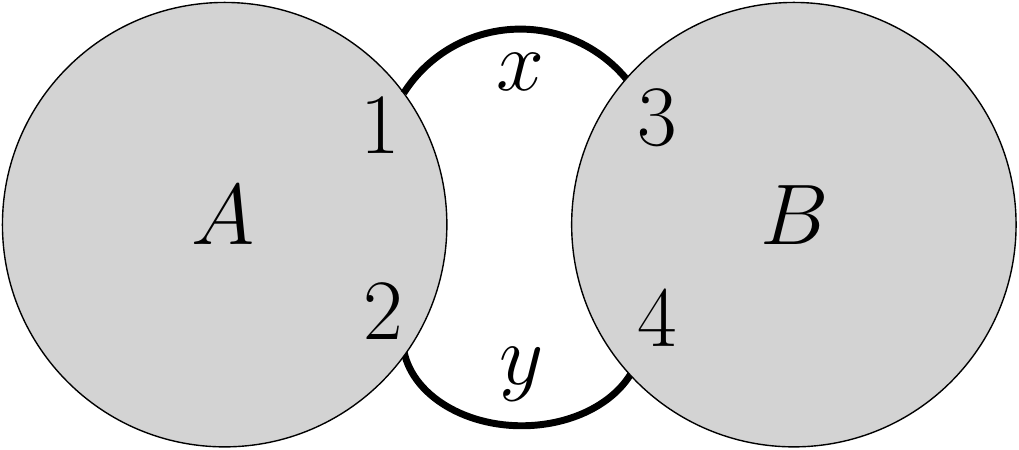}
\end{center}

With $ G $ as drawn we see that
$$
G = (x+y)AB + \Ais B + A \Bef
$$
and
$$
G_x = y AB + \Ais B + A \Bef.
$$

To get $G_x$ from $G$ we contracted $x$ by setting $x=0$.  Notice that in $G$, the variables $x$ and $y$ only appear in the term $(x+y)$.  We also see that the occurrence of $(x+y)$ in $G$ corresponds exactly to the occurrence of $y$ in $G_x$.  Therefore we can recover $G$ from $G_x$ by making the replacement $ y \rightarrow x+y $.

To delete $ y $ we take the $ y $ derivatives of $ G $ and $ G_x $, and we likewise delete $ x $ from $ G $ and $ G_y $.  Then we have the identity 
$$ G^x = G^y = G^y_x = G^x_y = AB .$$
We will also use the contraction deletion relation
$$  G_y = x G_y^x + G_{yx} .$$

There are two cases to be considered: (1) the edge that condition 1 is being tested for belongs to the two edge cut-set; (2) it does not belong to the cut-set. 

%\emph{Proof of case 1.}

\begin{proof}[Proof of case 1]

Let $ y $ be the edge that we are testing for condition 1.
\begin{enumerate}
\item[$ (\Rightarrow) $] 
Suppose $ 1(G,y) $.  Then 
$$ G_y \in \left< \partial G^y \right> 
= \left< \partial G^y_x  \right> .$$
Hence
$$ G_{yx} = G_y - x G_y^x = G_y - x G_x^y \in \left< \partial G_x^y \right> .$$
\item[$ (\Leftarrow) $]
 Suppose $ 1(G_x,y) $.  Then 
$$ G_{yx} \in \left< \partial G_x^y \right> .$$
Then
$$ G_y 
= x G_y^x + G_{yx} 
= x G_x^y + G_{yx} 
\in \left< \partial G_x^y \right> 
= \left< \partial G^y \right> .$$
\end{enumerate}
\renewcommand{\qedsymbol}{}
\end{proof}

\begin{proof}[Proof of case 2]

Let the edge $ e $ that we are testing for condition 1 belong to either $A$ or $B$, and wherever the $ a_i $ appear below, let them range over the edge variables not equal to $ x $, $ y $, or $ e $.
\begin{enumerate}

	\item[$ (\Rightarrow) $] 
	Suppose $ 1(G, e) $, so that for some polynomials $P_i,Q,R$ we have
	$$ G_e = \sum_i P_i G^{e a_i} + Q G^{e x} + R G^{e y}. $$
	Now we use $ G^{ex} = G^{ey} = G^{ey}_x $, and let $ S = Q+R $, so that
	$$ G_e = \sum_i P_i G^{e a_i} + S G_x^{e y}. $$
	Then we set $ x=0 $ to get
	$$ G_{ex} = \sum_i P_i(x=0) G_x^{e a_i} + S(x=0) G_x^{e y}. $$
	
	\item[$ (\Leftarrow) $] 
	Suppose $ 1(G_x,e) $, so that for some polynomials $P_i,R$ we have
	$$ G_{ex} = \sum_i P_i G_x^{e a_i} + R G_x^{e y}. $$
	Then, by the replacement $ y \rightarrow x+y $ we recover $ G_e $ from $ G_{ex} $ on the left hand side and we recover $ G^{e a_i} $ and $ G^{e y} $ on the right hand side. Therefore
	$$ G_e = \sum_i P_i(y \rightarrow x+y) G^{e a_i} + R(y \rightarrow x+y) G^{e y}. $$
\end{enumerate}
\renewcommand{\qedsymbol}{}
\end{proof}
\end{proof}

\medskip

Next we will define some special classes of graphs that we will use.

\begin{definition}
The \emph{wheel with $n$ spokes} is the graph with $n+1$ vertices consisting of a cycle of length $n$ along with an additional vertex which is adjacent to all the vertices of the cycle.  The edges of the cycle are called \emph{rim edges} while the other edges are called \emph{spoke edges}.
\end{definition}

\begin{definition}
A \emph{source-terminal graph} is a graph $G$ with two distinct marked vertices $s,t \in V(G)$. 

If $G$ and $H$ are two source-terminal graphs then we can define their \emph{parallel join} as being the source-terminal graph $G\paral H$, 
which is the disjoint union of $G$ and $H$ with the sources and terminals identified and with these two vertices forming the source and terminal of $G\paral H$.  

If $G$ and $H$ are two source-terminal graphs then we can define their \emph{series join} as being the source-terminal graph $G \bolt H$ which 
is the disjoint union of $G$ and $H$ with the source of $H$ identified with the terminal of $G$, with the source of $G$ becoming the source of $G\bolt H$, and the terminal of $H$ becoming the terminal of $G\bolt H$.  
\end{definition}

\begin{definition}
We take a series-parallel graph to be a source-terminal graph $G$ such that $G$ is either

\begin{enumerate}

\item
$G = K_2$.

\item
$G$ is the parallel join of two series-parallel graphs $H,H'$, i.e. $G := H\paral H'$.

\item
$G$ is the series join of two series-parallel graphs $H,H'$, i.e. $G := H \sjoin H'$.

\end{enumerate}
\end{definition}

In the case of series-parallel graphs the Kirchhoff polynomial for the graph with the two terminals identified is particularly important and so we will use the following notation, 
\[
\breaker{G} = \overset{st}{\kirk{G}}
\]
for any source-terminal graph $G$.

If we interpret $s,t \in V(G)$ as the only vertices with external edges then we recover $\breaker{G}$ as the second Symanzik polynomial.  That is 
\[
\breaker{{G}} = \sum_{T_1,T_2} \prod_{e \notin T_1 \cup T_2} x_e
\] 
Where $T_1,T_2$ are trees, $s \in T_1$, $t \in T_2$, $T_1 \cap T_2 = \emptyset$ and $V(G) \ssq T_1 \cup T_2$.  Call spanning forests of 2 trees with these properties \emph{spanning-st-forests}.

The following are all reformulations of the 1 and 2 vertex cut formulas for the Kirchhoff polynomial as applied to the series and parallel operations.
\begin{lemma} \label{Psi-P identities}
Let $H,H'$ be source-terminal graphs. Then
\begin{enumerate}[(a)]

\item
$\kirk{H\paral H'} = \kirk{H}\breaker{H'} + \breaker{H}\kirk{H'}$

\item
$\breaker{H\paral H'} = \breaker{H}\, \breaker{H'}$

\item
$\kirk{H \bolt H'} = \kirk{H}\kirk{H'}$

\item
$\breaker{H \bolt H'} = \breaker{H}\kirk{H'} + \kirk{H}\breaker{H'}$

\item
$\deg(\breaker{H}) = \deg(\kirk{H}) + 1$ and $\deg(\breaker{H}) > 0$.
\end{enumerate}
\end{lemma}

%\begin{proof}
%In each statement one takes the pertinent forest $F$ and restricts it to the components $H,H'$. Each restriction is either a spanning-st-forest or a tree depending on the context. The key is that the source terminal pair marks a 1-cut vertex in $H \bolt H'$ or marks a 2-cut pair in $H\paral H'$.
%\end{proof}

Any series-parallel graph has a natural recursive structure which we can capture in a tree.
\begin{definition}\label{def upsilon}
For any series-parallel graph $G$ we may associate to it a (not necessarily unique) decomposition tree $\Upsilon$, which is the rooted tree whose leaves represent the edges of $G$ and whose interior vertices represent the operations $\bolt, \paral$ used in the construction of $G$; the root vertex corresponds to the last operation used in the construction. Conversely, any such tree uniquely defines a series-parallel graph. 

The \emph{$\Upsilon$-dual} is the series-parallel graph associated to the decomposition tree $\Upsilon^\vee$ obtained by exchanging every $\bolt$ with a $\paral$ and vice-versa. Finally, $\height(\Upsilon)$ is the height of $\Upsilon$ as a rooted tree.
\end{definition}

%%%%%%%%%%%%%%%%%%%%%%%%%%%%%%%%%%%%%%%%%%%%%%%%%%%%%%%%%%%%%%%%%%%%%%%%%%%%%%%
\section{Multiple edges}

There are several interesting results concerning parallel edges and condition 1. Aluffi in \cite{Acond} showed what we give as Proposition
\ref{e parellel} 
to prove that the Chern class obeys a multiple edge formula.  Extending beyond his work, Propositions 
\ref{2 imp 1} and 
\ref{2 imp 3} 
imply that where there is a pair of parallel edges, adding a third parallel edge or more has no effect on condition 1, and if there are three or more edges, deleting all of them except two has no effect.  In the context of condition 1 one can look at a multigraph as a simple graph with two types of edges: the single edge and the multiple edge.  

%As well, Propositions 
%\ref{2 imp 1} and 
%\ref{2 imp 3} 
%further narrow the class of graphs that would need to be checked to determine the condition for the regular edges of all graphs.  Now this class is only the 2-connected graphs with no 2 edge cut-sets, and with parallel edges at most in pairs.  

\begin{prop}[Aluffi, 2011]\label{e parellel}
If $ e $ is a regular edge and $e$ has at least one other edge parallel to it in $ G $ then $ 1(G,e) $ is true.  

% % % % % % % % % % %
% % D R A W I N G % %
% % % % % % % % % % %
%\begin{center}
%\includegraphics[width=0.28\textwidth]{2ediag.pdf}
%\end{center}

\end{prop}

The main idea for the proof of Proposition \ref{e parellel} is Euler's homogeneous function theorem, see \cite{Acond}.

%For completeness and as a warm up for our later results we give two proofs of this result.
%\begin{proof}
%Let $ x $ be parallel to $ e $.  Then $ G_e = x G_e^x $ and also $ G_e^x = G_x^e $.  First suppose that $ x $ is not a bridge of $ G^e $.
%$$
%G^e_x = G^e - x G^{ex} \in \left< \partial G^e \right>
%$$
%Therefore
%$$
%G_e = x G_e^x \in \left< \partial G^e \right> .
%$$
%Next suppose $ x $ is a bridge of $ G^e $. 
 
% % % % % % % % % % %
% % D R A W I N G % %
% % % % % % % % % % %
%\begin{center}
%\includegraphics[width=0.28\textwidth]{isBrg.pdf}
%\end{center}

%Then $ G_x^e = G^e $.  So
%$$
%G_e = x G_e^x = x G_x^e = x G^e \in \left< \partial G^e \right> .
%$$
%\end{proof}

%\begin{proof}[alternative proof]
%For a graph, as in the diagram, with $e$ and $x$ parallel, $1(G,e)$ reads

%\begin{equation}\label{alt2e}
%G_e = x \Ais = \sum_i P_i (x\Aa + \Aisa ) +Q A
%\end{equation}
% % % % % % % % % % %
% % D R A W I N G % %
% % % % % % % % % % %
%\begin{center}
%\includegraphics[width=0.28\textwidth]{Aex.pdf}
%\end{center}

%We have, by Proposition 
%\ref{thm euler},
%$$
%A = \frac{1}{l-2}\sum_i a_i \Aa; \quad \Ais = \frac{1}{l-1}\sum_i a_i \Aisa,  
%$$
%where $l$ is the loop number of $G$, making $A$ degree $l-2$ and $\Ais$ degree $l-1$.
%Looking closely at Equation 
%\ref{alt2e}, 
%we will get what we get what we want on the LHS and have cancellation on the RHS if we choose 
%$$ P_i = \frac{x a_i}{l-1}; \quad Q = \frac{-x^2(l-2)}{l-1}. $$
%\end{proof}

The next three propositions have a common set-up. 
% % % % % % % % % % %
% % D R A W I N G % %
% % % % % % % % % % %
\begin{center}
\includegraphics[width=0.44\textwidth]{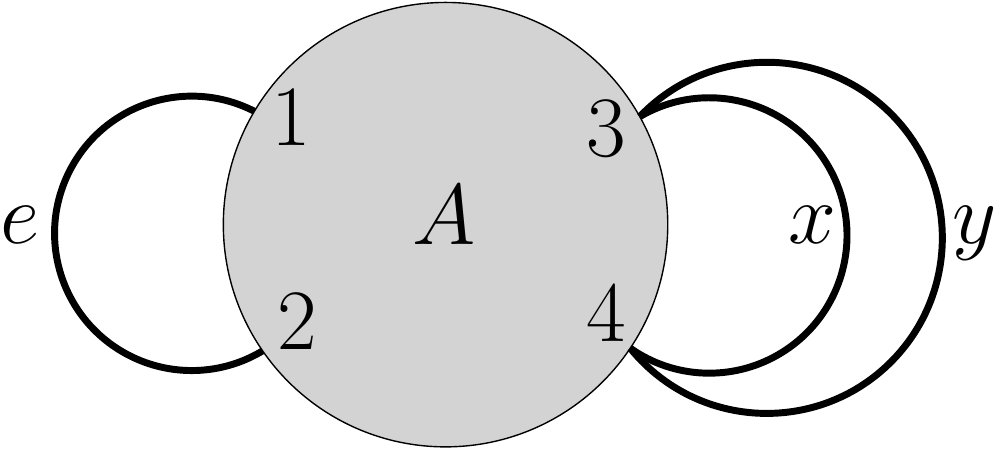}
\end{center}
For a graph with edge $ e $ connected to $ A $ at vertices $ 1, 2 $ and parallel edges $ x, y $ connected to $ A $ at vertices $ 3, 4 $, the condition $ 1(G,e) $ reads as
\begin{equation}\label{parEq}
\begin{aligned}
G_e = xy\Ais + (x+y)\Aisef &= \sum_{i}P_i \left(xy\Aa + (x+y)\Aefa \right) \\
&+ Q(yA+ \Aef) + R(xA+ \Aef).\\
\end{aligned}
\end{equation}
As the ideal generated by $\kirk{G_e}$ is homogeneous we may assume that the $ P_i, Q, R $ are homogeneous degree 2 polynomials in the variables $x, y, a_i$ (cf. \cite[Section 4.2]{fulton1989algebraic}). Let us write these polynomials as the solution set
$$ P_i,\quad Q,\quad R $$
Equation \ref{parEq} has a symmetry which is quite useful.  Let $ \tau_{xy} $ be the operation where $ x $ and $ y $ are swapped.  Applying $ \tau_{xy} $ to both sides gives
\begin{equation}\label{parEqxy}
\begin{aligned}
xy\Ais + (x+y)\Aisef &= \sum_{i}\tau_{xy}P_i \left( xy\Aa + (x+y)\Aefa \right) \\
&+ \tau_{xy}Q(xA+ \Aef) + \tau_{xy}R(yA+ \Aef).\\
\end{aligned}
\end{equation}
The left hand side is invariant under this operation. Thus we can sum equations \ref{parEq} and \ref{parEqxy}, and divide by 2, yielding the polynomials
$$ 
\widetilde{P_i} = \frac{P_i + \tau_{xy}P_i}{2},\quad \widetilde{Q} = \frac{Q + \tau_{xy}R}{2},\quad \widetilde{R} = \frac{R + \tau_{xy}Q}{2}
$$
which satisfy Equation \ref{parEq}. Hence from any given solution set $ P_i, Q, R $, one can construct $ \widetilde{P_i}, \widetilde{Q}, \widetilde{R} $ such that
$$ 
\tau_{xy}\widetilde{P_i} = \widetilde{P_i}, \quad 
\tau_{xy}\widetilde{Q} = \widetilde{R}, \quad
\tau_{xy}\widetilde{R} = \widetilde{Q}
.$$
Focusing only on the $ x,y $ dependence, in order to satisfy these relations and Equation \ref{parEq} the $ \widetilde{P_i}, \widetilde{Q}, \widetilde{R} $ must have the forms
$$
\begin{aligned}
\widetilde{P_i} &= p_{0 i} + (x+y)p_{1 i} + (x^2+y^2)p_{2 i} + xy p_{11 i}\\
\widetilde{Q} &= q_0 + x q_{10} + y q_{01} + x^2 q_{20} + y^2 q_{02} + xy q_{11}\\
\widetilde{R} &= q_0 + x q_{01} + y q_{10} + x^2 q_{02} + y^2 q_{20} + xy q_{11}.\\
\end{aligned}
$$
Substituting these expressions into Equation \ref{parEq} and collecting terms we get the equation:
\begin{center}
\begin{tabular}{l l l l}
$xy\Ais +(x+y)\Aisef =$ & $ xy\Sigma_i p_{0 i} \Aa$			& $+(x+y)q_0 A$					&$+2q_0 \Aef$\\
 			& $+ (x^2y+xy^2)\Sigma_ip_{1 i}\Aa	$			& $+2xy q_{1 0} A$				&$+(x+y) q_{10} \Aef$\\
			& $+ (x^3y + xy^3)\Sigma_ip_{2 i}\Aa$			& $+(x^2+y^2) q_{01} A$ 		&$+(x+y) q_{01} \Aef$\\
 		& $+ x^2y^2\Sigma_ip_{11 i}\Aa$					& $+ (x^2y+xy^2) q_{20} A$ 		&$+ (x^2+y^2) q_{20} \Aef$\\
 		&$+ (x+y)\Sigma_ip_{0 i}\Aefa	$				& $+ (x^3+y^3) q_{02} A $		&$+ (x^2+y^2) q_{02} \Aef$\\
			& $+ (x^2+y^2 + 2xy)\Sigma_ip_{1 i}\Aefa$		& $+ (x^2y+xy^2) q_{11} A $ 	&$+ 2xy q_{1 1} \Aef.$\\
 			& $+ (x^3+y^3+ x^2y+xy^2)\Sigma_ip_{2 i}\Aefa$	&  								&\\
			& $+ (x^2y+xy^2)\Sigma_ip_{11 i}\Aefa$			&  								&\\
\end{tabular}
\end{center}

This is an equation as polynomials in $x$ and $y$, so we can get a list of equations by equating coefficients. This gives
$$
\begin{aligned}
&[x^0y^0]: \quad 0 = 2q_0 \Aef\\
&[x^1y^0]: \quad \Aisef = \Sigma_ip_{0 i}\Aefa + q_0 A + (q_{10}+q_{01})\Aef\\
&[x^1y^1]: \quad \Ais = \Sigma_ip_{0 i}\Aa + 2\Sigma_ip_{1 i}\Aefa + 2q_{10}A+2q_{11}\Aef\\
&[x^2y^0]: \quad 0 = \Sigma_ip_{1 i}\Aefa + q_{01}A + (q_{20}+q_{02})\Aef\\
&[x^3y^0]: \quad 0 = \Sigma_ip_{2 i}\Aefa + q_{02}A\\
&[x^2y^1]: \quad 0 = \Sigma_ip_{1 i}\Aa + \Sigma_ip_{2 i}\Aefa + \Sigma_ip_{11 i}\Aefa + (q_{20}+q_{11})A\\
&[x^3y^1]: \quad 0 = \Sigma_ip_{2 i}\Aa\\
&[x^2y^2]: \quad 0 = \Sigma_ip_{11 i}\Aa.\\
\end{aligned}
$$
 By following the steps in this derivation we can determine a set of polynomials satisfying these eight equations when Equation \ref{parEq} is satisfied.  Conversely, given a solution set for these equations we can obtain a solution set for Equation \ref{parEq}.  In this sense the eight equations above are equivalent to Equation 
\ref{parEq}.  
%Due to symmetry, any other equation that can be extracted from the coefficients will be equal to one of these eight. A weaker but computationally useful way of stating this is that the eight polynomials above are a basis for the ideal in the coefficient ring generated by the relations defining condition 1. 

We take the equation $[x^1y^1]$, and use $ [x^2y^0] $ to substitute out $ \Sigma_ip_{1 i}\Aefa $, and obtain the equation
$$[x^1y^1]^*: \quad \Ais = \Sigma_ip_{0 i}\Aa + 2(q_{10} - q_{01})A+2(q_{11} - q_{20} - q_{02})\Aef.\\$$

As a consequence of the $[x^0y^0]$ equation we have $q_0 = 0 $. (We see $\Aef$ is not zero so $q_0$ must be.)  Hence we can remove the $ q_0 A $ term from the $[x^1y^0]$ equation yielding 
$$[x^1y^0]^*: \quad \Aisef = \Sigma_ip_{0 i}\Aefa + (q_{10}+q_{01})\Aef. $$

These two expressions will be used in the next three proofs.

\begin{prop}\label{contract2}
If $x, y$ are parallel edges then $ 1(G,e) \Rightarrow 1(G^y_x,e)$.  

% % % % % % % % % % %
% % D R A W I N G % %
% % % % % % % % % % %
%\begin{center}
%\includegraphics[width=0.28\textwidth]{2impCdiag.pdf}
%\end{center}

\end{prop}

%This is the smallest case of cycle contraction.

\begin{proof}
Suppose $1(G,e)$.  Then Equation \ref{parEq} has a solution set.  
We look at the $ [x^1y^0]^* $ equation:
$$ 
\Aisef = \Sigma_ip_{0 i}\Aefa + (q_{10}+q_{01})\Aef. 
$$
The left hand side is  $ \Aisef = G^{y}_{x e} $.  
We want to show that it is in the ideal $ \langle \partial G^{ye}_x \rangle $.  
Here $ G^{y e}_x = \Aef $.   
By Theorem 
\ref{thm euler} 
we have $ \Aef \in \langle \partial \Aef \rangle $.  
The $ \Aefa $ are partials of $\Aef$ themselves.
Thus $ \Aisef \in \langle \partial \Aef \rangle $, so $1(G^y_x,e)$ is satisfied.
\end{proof}

\begin{prop}\label{2 imp 1}
If $x, y$ are parallel edges then $ 1(G,e) \Rightarrow 1(G^y,e)$.

% % % % % % % % % % %
% % D R A W I N G % %
% % % % % % % % % % %
%\begin{center}
%\includegraphics[width=0.28\textwidth]{2imp1diag.pdf}
%\end{center}

\end{prop}

\begin{proof}
First, note that $ G^{y e}= xA + \Aef $. 
Suppose $1(G,e)$ and create the equation $ x[x^1y^1]^* + [x^1y^0]^*$.
\begin{equation}\label{2eq}
x\Ais + \Aisef = \Sigma_ip_{0 i}(x\Aa + \Aefa) + 2x( q_{10} -  q_{01})A + [2x( q_{11}- q_{20}- q_{02}) + ( q_{10}+ q_{01})]\Aef.
\end{equation}
The left hand side of Equation 
\ref{2eq} 
is $ G^y_e $. $ A$ and $\Aef $ are both in the ideal of partial derivatives of $ (xA + \Aef ) $ because $ A = (\frac{\partial}{\partial x}) (xA + \Aef )$ and $ \Aef = (xA + \Aef ) - xA$. 

Hence $ G^y_e $ is in the ideal $ \langle \partial G^{ye} \rangle $, therefore we have $1(G^y,e)$.
\end{proof}

\begin{prop}\label{2 imp 3}
If $x, y$ are parallel edges and $ G\cup z $ is obtained by adding an edge $z$ parallel to x and y,  then $ 1(G,e) \Rightarrow 1(G\cup z,e)$.

% % % % % % % % % % %
% % D R A W I N G % %
% % % % % % % % % % %
%\begin{center}
%\includegraphics[width=0.28\textwidth]{2imp3diag.pdf}
%\end{center}

\end{prop}

\begin{proof}
This proof is very similar to the proof of Proposition \ref{2 imp 1}. We have 
	\[
		(G\cup z)^e = xyzA + (yz+xz+xy)\Aef.  
	\]
The following polynomials are in the ideal $\langle \partial (G\cup z)^e \rangle$:
$$
yz\Aef, \quad xz\Aef, \quad xy\Aef, \quad xyzA.
$$
This is because 
	\[ 
	yz\Aef = (xyzA + (yz+xz+xy)\Aef) - x(\partial/\partial x)(xyzA + (yz+xz+xy)\Aef).
	\] 
A similar argument works for $xz\Aef$, $xy\Aef$. Additionally, we have
	\[ 
	xyzA = (xyzA + (yz+xz+xy)\Aef) - yz\Aef - xz\Aef - xy\Aef.
	\]

Suppose $1(G,e)$. We make the equation $ xyz[x^1y^1]^* + (yz+xz+xy)[x^1y^0]^* $:
$$
\begin{aligned}
xyz\Ais + (yz+xz+xy)\Aisef &= \Sigma_ip_{0 i}(xyz\Aa + (yz+xz+xy)\Aefa) \\
&+ 2xyz (q_{10} - q_{01})A+2xyz(q_{11} - q_{20} - q_{02})\Aef\\
&+ (yz+xz+xy)(q_{10}+q_{01})\Aef.\\
\end{aligned}
$$
It follows that $ 1(G\cup z,e) $.
\end{proof}

\section{Wheel graphs}

We can give a full characterization of which edges of wheel graphs satisfy condition 1.  Specifically for wheels with more than 3 spokes, condition 1 is false for all rim edges and true for all spoke edges.

We need a few lemmas to obtain the results.
The next lemma is closely related to Proposition 
\ref{2edge} 
and says that if two edges form a two edge cut set of $G^e$ then contracting one of them is condition 1 preserving.  
\begin{lemma}\label{lemma r}
If a graph $G$ has the form of the graph on the left below, 

% % % % % % % % % % %
% % D R A W I N G % %
% % % % % % % % % % %
\begin{center}
\includegraphics[width=0.54\textwidth]{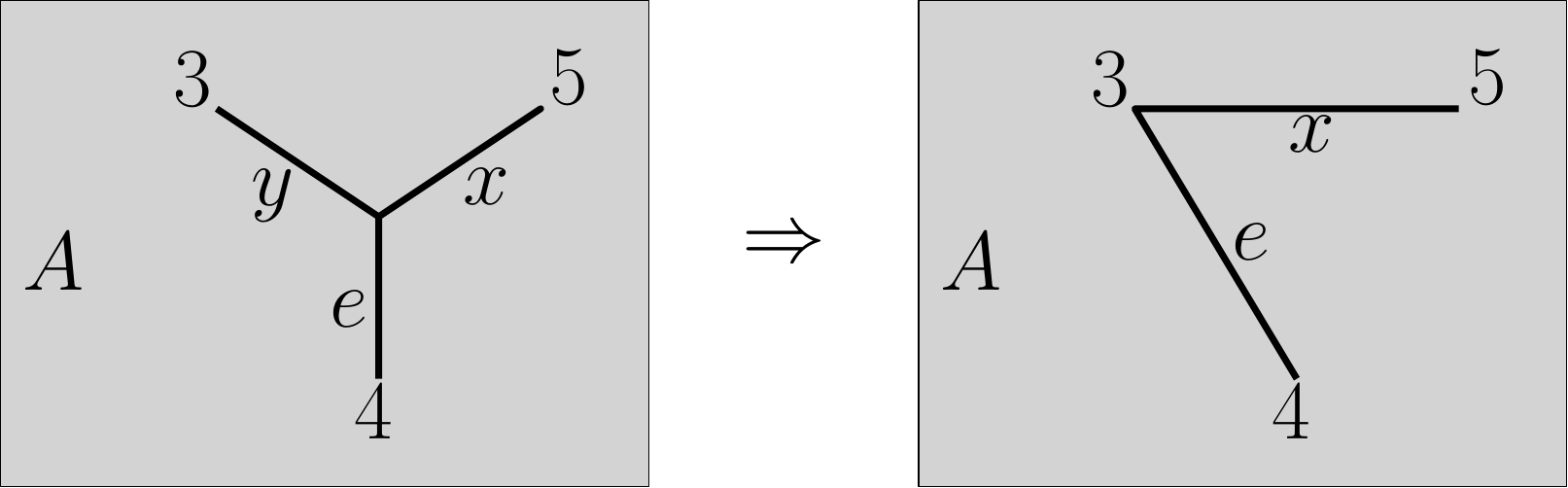}
\end{center}

with the equation 
$$
G = (ey+ex+xy)A + x\Aef +y\Afg + e\Aeg,
$$
then $1(G,e) \Rightarrow 1(G_y,e)$.  
\end{lemma}
Note that here the central vertex connecting $x$, $y$, and $e$ is only connected to the rest of the graph $A$ through $x$, $y$, and $e$.
\begin{proof}
The equation for $ 1(G,e) $ is 
\begin{equation}\label{w1}
xyA + x\Aef + y\Afg = \sum_i P_i ((x+y)\Aa + \Aega) +Q A + R A.
\end{equation}
The equation for $ 1(G_y,e) $ is 
\begin{equation}\label{w2}
xA + \Aef = \sum_i S_i (x\Aa + \Aega) +T A.
\end{equation}
Set $ y=0 $ in Equation \ref{w1}.  Then we have a solution for Equation \ref{w2} by selecting 
$$
S_i = P_i(y=0); \quad T = Q(y=0) + R(y=0).
$$
\end{proof}

\begin{prop}\label{rim prop}
For the wheel graphs with $ n>3 $ sides, condition 1 is false for all rim edges.
\end{prop}

\begin{proof}
This will be proved by induction.  Let $ W_n $ be a wheel graph with $ n $ sides and let $ r $ be any rim edge.  We will show that for $ n>3 $ we have
\begin{equation}\label{w3}
1(W_n,r) \Rightarrow 1(W_{n-1},r).
\end{equation}
%The $ n>3 $ is assumed because the smallest wheel graph of interest has 3 sides.  
Assuming $ n>3 $ guarantees that $ r $ is a regular edge of $ W_{n-1} $.

% % % % % % % % % % %
% % D R A W I N G % %
% % % % % % % % % % %
\begin{center}
\includegraphics[width=0.84\textwidth]{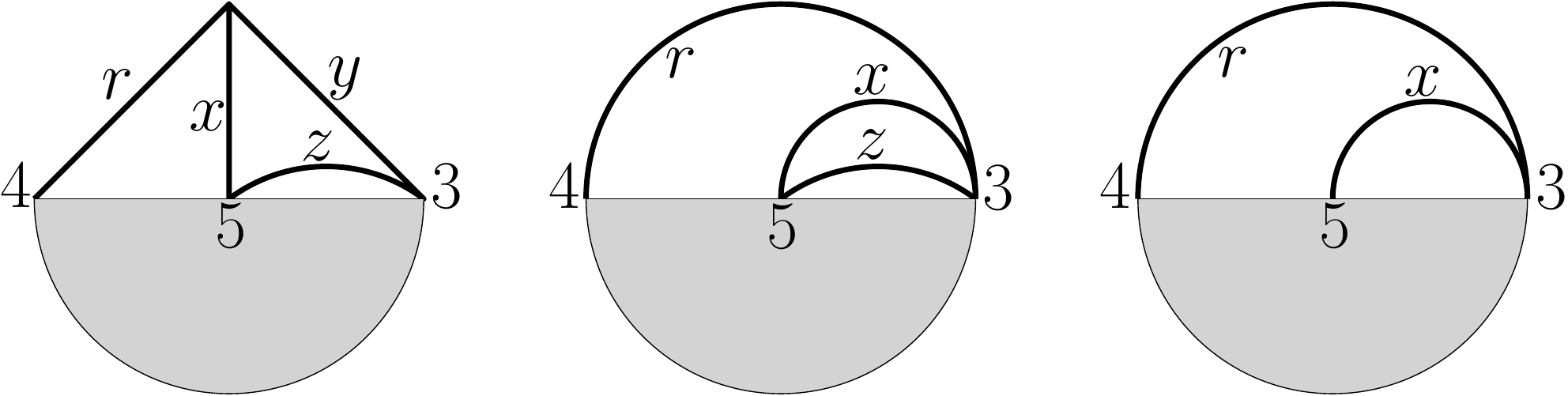}
\end{center}

Now assume $ 1(W_n,r) $.  Then, looking at the diagram above we see that $ r $ satisfies the conditions of Lemma 
\ref{lemma r}, 
so we can contract the rim edge $ y $ beside $ r $, to get $ 1(W_{n \: y},r) $.  

But then $ x $ and $ z $ are parallel in $ W_{n \: y} $, so by Proposition 
\ref{2 imp 1} 
we can delete $ z $ to get $ 1(W_{n \: y}^{ \: z},r) $.  But this is just condition 1 for the wheel graph with one less side, as $ W_{n \: y}^{ \: z} = W_{n-1} $.  Therefore we have proved Equation 
\ref{w3}.  
Then taking the contrapositive,
$$
\neg 1(W_{n-1},r) \Rightarrow \neg 1(W_n,r),
$$
for $ n>3 $.

The base cases can be verified explicitly.
%Using the methods discussed in Section \ref{comp}, the result $ \neg 1(W_n,r) $ has been computed for $ n = 4 $ to $ 8 $. 
 By induction with $ n = 4 $ as the base case we obtain $ \neg 1(W_n,r)$ for $n>3$.
\end{proof}

\begin{lemma}\label{lemma s}
$ 1(G,e) $ is true for all graphs having the form

% % % % % % % % % % %
% % D R A W I N G % %
% % % % % % % % % % %
\begin{center}
\includegraphics[width=0.24\textwidth]{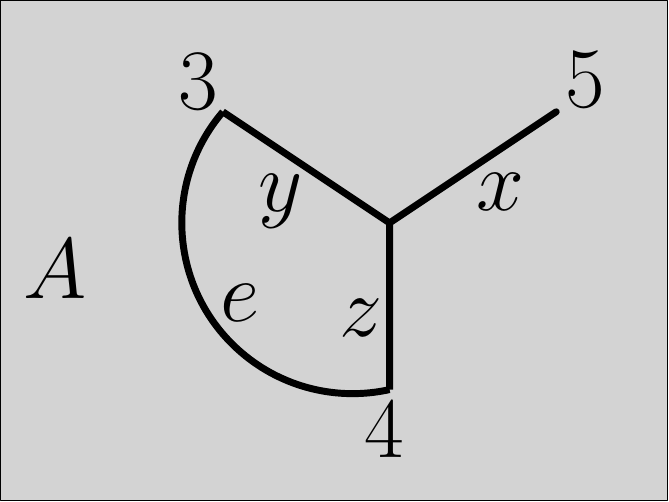}
\end{center}

with the equation 
$$
G = e[(yz+xz+xy)A + x\Aef, + y\Afg + z\Aeg + \Aefg] + (yz+xz+xy)\Aef + (y+z)\Aefg.
$$
\end{lemma}

\begin{proof}
The equation $ 1(G,e) $ reads 
	\begin{equation}\label{eq of good}
		\begin{aligned} 
			&(yz+xz+xy)\Aef + (y+z)\Aefg  \\
			= \ &\sum_i P_i [ (yz+xz+xy)\Aa + x\Aefa + y\Afga + z\Aega + \Aefga ] \\
			&+ Q[(y+z)A + \Aef]
			+ R[(x+z)A + \Afg]
			+ S[(x+y)A + \Aeg].
		\end{aligned} 
	\end{equation}
Now we will make use of Proposition 
\ref{thm euler} 
%in the same way that we did for the alternative proof of Proposition 
%\ref{e parellel} 
to try to guess what the solution polynomials to Equation 
\ref{eq of good} 
are.  

The $ P_i $ are the logical polynomials to guess first because $ \Aefg $ occurs with them and nowhere else on the right hand side of Equation 
\ref{eq of good}.  
Then to equate the $ \Aefg $ terms, we will try 
$$
P_i = \frac{(y+z) a_i}{l - 1}
$$
The loop number\footnote{In graph theory language the loop number of a graph is the dimension of the cycle space of the graph; the term `loop number' comes from physics.} of $ G $, which we denote by $l$, is chosen as a reference value. Here, $A$ (with 4 fewer edges and 1 fewer vertex than $G$) has loop number $ l - 3 $, $ \Aef $ (with $1$ fewer vertex than $A$) has loop number $ l - 2 $, and $ \Aefg $ has loop number $ l - 1 $.

With this choice of the $ P_i $, the $ \sum_i P_i [...] $ in Equation 
\ref{eq of good} 
evaluates to 
$$
\frac{y+z}{l-1}[ (yz+xz+xy)(l-3)A + x(l-2)\Aef + y(l-2)\Afg + z(l-2)\Aeg + (l-1)\Aefg ].
$$

Next, we show that the $ \Aef $ terms work out.  Notice that 
$$
\frac{l-2}{l-1} = 1 - \frac{1}{l-1}.
$$
Then we endeavour to satisfy 
$$
(yz+xz+xy)\Aef = x(y+z)\left( 1 - \frac{1}{l-1}\right)\Aef + Q \Aef,
$$
which can be accomplished by choosing
$$
Q = yz + \frac{x(y+z)}{l-1}.
$$

Finally, to make $ \Afg $ and $ \Aeg $ work out we choose
$$
R = \frac{-y(y+z)(l-2)}{l-1}; \quad S = \frac{-z(y+z)(l-2)}{l-1}.
$$

Then with these choices, by construction, all of the terms involving $ \Aef, \Afg, \Aeg $, and $ \Aefg $ equate on both sides of $ P_i, Q, R, S $.  
Finally, then, as the reader can verify, the terms involving $A$ all cancel which gives a valid solution set $P_i, Q, R, S$.

%But we have not yet considered the terms involving $ A $.  If only they were to cancel on the RHS of Equation \ref{eq of good}, we would now have a valid solution set $ P_i, Q, R, S $.

%Fortunately, as the reader can verify, the terms involving $ A $ do indeed all cancel giving the proposition.

\end{proof}

\begin{prop}\label{spoke prop}
For the wheel graphs with $ n\geq 3 $ sides, condition 1 is true for all spoke edges.
\end{prop}

\begin{proof}
This is an immediate consequence of Lemma 
\ref{lemma s}.  
Let $ s $ be a spoke edge of $ W_n $.  

% % % % % % % % % % %
% % D R A W I N G % %
% % % % % % % % % % %
\begin{center}
\includegraphics[width=0.24\textwidth]{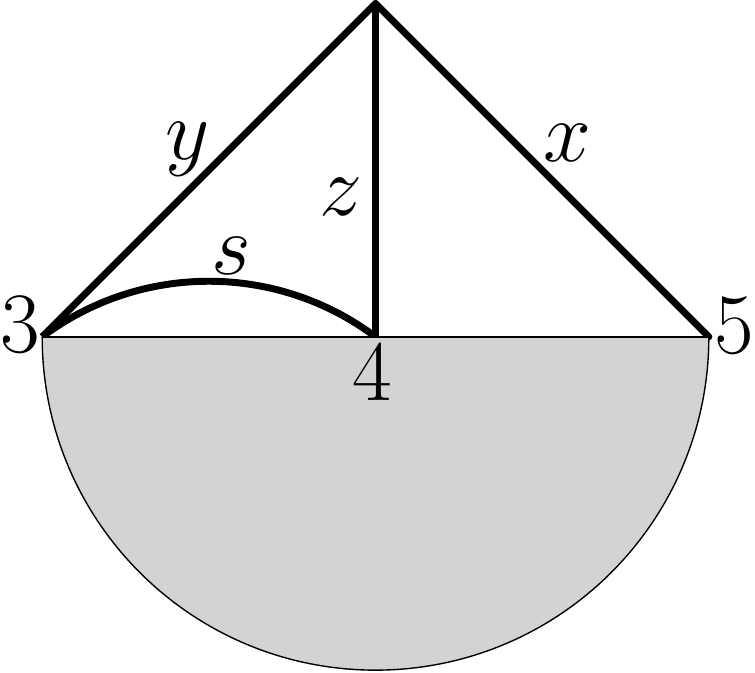}
\end{center}

With the edges as labelled we see that $ (W_n, s) $ has the form specified by the lemma, therefore $ 1(W_n, s) $ is true.
\end{proof}

The cancellation of terms involving $A$ in Lemma~\ref{lemma s} seems like something of a minor miracle.  It is suggestive that making a $\Delta$ to $Y$ transformation in the graph ought to preserve condition 1.  In fact this possibility is what suggested Lemma~\ref{lemma s} to us.  See Section~\ref{sec conclusion} for further discussion of this point.

%There is another possible result called "Delta implies Y", or $ \Delta \Rightarrow Y $.  So far I have not been able to prove $ \Delta \Rightarrow Y $, but never the less I believe it is probably true, at first because Karen suggested it might be worth looking into and I could not find a counter-example, and now because of Lemma 
%\ref{lemma s}. 
 
% % % % % % % % % % %
% % D R A W I N G % %
% % % % % % % % % % %
%\begin{center}
%\includegraphics[width=0.44\textwidth]{WDimpY1.pdf}
%\end{center}

%Delta implies Y pertains to a a graph $ G_{\Delta} $, as on the left in the above diagram, with 3 edges forming a $ \Delta $ shape and a distinct fourth edge $ e $.  It says that, if $ 1(G_{\Delta}, e) $ is true then we can replace the $ \Delta $-forming edges with edges forming a $ Y $, to make the graph $ G_Y $, and $ 1(G_{Y}, e) $ will be true.

%\begin{conj}\label{DimpY}
%With $ G_{\Delta} $ and $ G_Y $ as described, 
%$$
% 1(G_{\Delta}, e)  \Rightarrow 1(G_{Y}, e).
%$$
%\end{conj}

%Now toward answering what motivated Lemma 
%\ref{lemma s}, 
%consider what happens if $ e $ is parallel to one of the $ \Delta $-forming edges.  
%Then $ 1(G_{\Delta}, e) $ is true by Proposition \ref{e parellel}.  Then $ \Delta \Rightarrow Y $ would require that $ 1(G_{Y}, e) $ also be true, and in this special case $ G_{Y} $ is the graph in  Lemma 
%\ref{lemma s}.  
%That is what led me to search for this result, and to find a suspiciously fortuitous cancellation allowing the solution to be realized.

%%%%%%%%%%%%%%%%%%%%%%%%%%%%%%%%%%%%%%%%%%%%%%%%%%%%%%%%%%%%%%%%%%%%%%%%%%%%%%%%%
\section{Simultaneous combinations in series-parallel graphs}

Condition 1 itself is not ideally suited to the recursive constructions involved in building series parallel graphs.  Specifically, one frequently wants to combine expansions of the polynomials $\kirk{G}$ and $\breaker{G}$ but lacks any information on how the coefficient polynomials relate.  To work around this we first consider a stronger condition where the coefficients are controlled.

\begin{definition} \label{simultaneous combination} 
	Let $G$ be a source-terminal (series-parallel) graph and $e_1 \in G$ an edge. We say that \emph{simultaneous combination} holds for $(G,e_1)$, or $S(G,e_1)$ holds, if there are polynomials $A_j,B,C$ such that
	\allstar{
		\kirk{G} &= \sum A_j \dkirk{G}{1j} + B \cdot\dkirk{G}{1} \\
		\breaker{G} &= \sum A_j \dbreaker{G}{1j} + C \cdot \dbreaker{G}{1}
	}
	We say that \emph{simultaneous combination} holds for $G$, or $S(G)$ holds, if $S(G,e)$ holds for all $e \in G$.
\end{definition}

Note that if $S(G,e)$ holds, then by Euler's Theorem we can choose either $B = 0$ or $C = 0$ in the statement of Definition~\ref{simultaneous combination}.  The freedom to choose which is $0$ is quite handy and explains why this symmetric, albeit redundant, definition was chosen.

\begin{prop}
	$S(G,e)$ implies condition 1 holds for $e \in G$.
\end{prop} 

\begin{proof}
	Immediate from the definitions.
\end{proof}

Simultaneous combinations are well behaved with respect to both series and parallel operations.
\begin{lemma} \label{S-lift}
	Let $H,H'$ be source-terminal graphs and let $e \in H$. If $S(H,e)$ then $S(H\paral H',e)$ and $S(H \bolt H',e)$.
\end{lemma}

The basic plan is to put the required linear combinations together using Euler's theorem.  This technique will be a theme in what follows, but since this is the first such argument we will go into detail.

\begin{proof}[Proof for $S(H\paral H',e)$]
	Let $e = e_1$, let $n := \deg(\breaker{H})$ , let $m := \deg(\breaker{H'})$, and let $G := H\paral H'$. Then
	\allstar{
		\kirk{G} &= \kirk{H}\breaker{H'} + \breaker{H} \kirk{H'} \\
		\breaker{G} &= \breaker{H}\,\breaker{H'}
	}
	so
	\allstar{
		\gen{\dkirk{G}{1j}} &= \gen{ \dkirk{H}{1j} \breaker{H'} + \dbreaker{H}{1j} \kirk{H'}, \dkirk{H}{1} \dbreaker{H'}{j} + \dbreaker{H}{1} \dkirk{H'}{j}} \\
		\gen{\dbreaker{G}{1j}} &= \gen{\dbreaker{H}{1j}\breaker{H'}, \dbreaker{H}{1} \dbreaker{H'}{j}}.
	}
	
	We note that $\deg(\breaker{H}) = \deg(\kirk{H}) + 1$, so by Euler's theorem the components of the vector
	\[
		\bbm
			(n-2) & (n-1) \\
			m & (m-1)
		\ebm
		\bbm
			\dkirk{H}{1} \breaker{H'} \\
			\dbreaker{H}{1}  \kirk{H'}
		\ebm
	\]
	are in $\gen{\dkirk{G}{1j}}$. Thus so is
	\[
		m\sum_{j \in H} x_j \dkirk{G}{1j} - (n-2)\sum_{j \in H'} x_j \dkirk{G}{1j} = (n+m-2) \dbreaker{H}{1}\kirk{H'}.
	\]
	We also have that
	\allstar{
		m\sum_{j \in H} x_j \dbreaker{G}{1j} - (n-2)\sum_{j \in H'} x_j \dbreaker{G}{1j} &= m(n-1) \dbreaker{H}{1} \breaker{H'} - (n-2)m \dbreaker{H}{1} \breaker{H'} \\
		&= m \dbreaker{H}{1} \breaker{H'}.
	}
	For $j \in H$ pick $A_j,C$ as in the statement of $S(H,e_1)$ (using the previously observed freedom to set $B=0$). Then for
	\[
		B_j := \piecewise{A_j + \frac{C}{n+m-2}\bra{mx_j}}{j \in H}{-\frac{(n-2)C}{n+m-2}x_j} 
	\]
	\[
		C' := \bra{1-\frac{m}{n+m-2}}C
	\]
	we verify that we have $S(H\paral H',e)$.
	\allstar{
		\sum_{j \in G} B_j \dkirk{G}{1j} = &\sum_{j \in H}\bra{A_j + \frac{C}{n+m-2}\bra{mx_j}}\dkirk{G}{1j} 
		+ \sum_{j \in H'} \bra{-\frac{(n-2)C}{n+m-2}x_j} \dkirk{G}{1j} \\
		= &\sum_{j \in H}A_j  \dkirk{G}{1j} +  \sum_{j \in H} \frac{C}{n+m-2}\bra{mx_j} \dkirk{G}{1j} - \sum_{j \in H'}\frac{(n-2)C}{n+m-2}x_j \dkirk{G}{1j} \\ 
		= &\sum_{j \in H}A_j  \dkirk{G}{1j} +  \frac{C}{n+m-2} \bra{ m \sum_{j \in H} x_j \dkirk{G}{1j} - (n-2)\sum_{j \in H'}x_j \dkirk{G}{1j}} \\
		= &\sum_{j \in H}A_j  \bra{\dkirk{H}{1j} \breaker{H'} + \dbreaker{H}{1j} \kirk{H'}} + C \cdot \dbreaker{H}{1} \kirk{H'}.
	}
	By $S(H,e_1)$ we can simplify the remaining expression
	\allstar{
		= & \sum_{j \in H}A_j\dkirk{H}{1j}\breaker{H'} +  \sum_{j \in H}A_j \dbreaker{H}{1j} \kirk{H'} + C \cdot \dbreaker{H}{1} \kirk{H'} \\
		= & \ \kirk{H}\breaker{H'} + \kirk{H'}(\breaker{H} - C \dbreaker{H}{1}) + C \cdot \dbreaker{H}{1} \kirk{H'} \\
		= & \ \kirk{G}.
	}
	
	\bigskip \noindent
	Meanwhile
	\allstar{
		\sum_{j \in G} B_j \dbreaker{G}{1j} = &\sum_{j \in H}\bra{A_j + \frac{C}{n+m-2}\bra{mx_j}}\dbreaker{G}{1j} 
		+ \sum_{j \in H'} \bra{-\frac{(n-2)C}{n+m-2}x_j} \dbreaker{G}{1j} \\
		= &\sum_{j \in H}A_j  \dbreaker{G}{1j} +  \sum_{j \in H} \frac{C}{n+m-2}\bra{mx_j} \dbreaker{G}{1j} - \sum_{j \in H'}\frac{(n-2)C}{n+m-2}x_j \dbreaker{G}{1j} \\ 
		= &\sum_{j \in H}A_j  \dbreaker{G}{1j} +  \frac{C}{n+m-2} \bra{ m \sum_{j \in H} x_j \dbreaker{G}{1j} - (n-2)\sum_{j \in H'}x_j \dbreaker{G}{1j}} \\
		= &\sum_{j \in H}A_j  \bra{\dbreaker{H}{1j} \breaker{H'}} + \frac{C}{n+m-2} \cdot m \dbreaker{H}{1} \breaker{H'}.
	}
	By $S(H,e_1)$ we can simplify the remaining expression
	\allstar{
		=  & \breaker{H'} (\breaker{H} - C \dbreaker{H}{1}) + \frac{mC}{n+m-2} \cdot  \dbreaker{H}{1} \breaker{H'} \\
		= & \ \breaker{G} - \bra{1-\frac{m}{n+m-2}}C \cdot  \dbreaker{G}{1}.
	} 
	So $S(G,e_1)$ holds.
\end{proof}

\begin{proof}[Proof for $S(H \bolt H',e)$]
	We use a similar argument for the series join. Let $e = e_1$, let $n := \deg{\breaker{H}}$, let $m := \deg{\breaker{H'}}$, and let $G := H \bolt H'$. Then
	\allstar{
		\kirk{G} &= \kirk{H} \kirk{H'} \\
		\breaker{G} &= \breaker{H}\kirk{H'} + \kirk{H} \breaker{H'}
	}
	so
	\allstar{
		\gen{\dkirk{G}{1j}} &= \gen{\dkirk{H}{1j}\kirk{H'}, \dkirk{H}{1} \dkirk{H'}{j}} \\
		\gen{\dbreaker{G}{1j}} &= \gen{ \dbreaker{H}{1j} \kirk{H'} + \dkirk{H}{1j} \breaker{H'}, \dbreaker{H}{1}\dkirk{H'}{j} + \dkirk{H}{1} \dbreaker{H'}{j}}.
	}
	We note that $\deg(\breaker{H}) = \deg(\kirk{H}) + 1$, so by Euler's theorem the coordinates of
	\[
		\bbm
			(n-1) & (n-2) \\
			(m-1) & m
		\ebm
		\bbm
			\dbreaker{H}{1} \kirk{H'} \\
			\dkirk{H}{1}  \breaker{H'}
		\ebm
	\]
	are in $\gen{\dbreaker{G}{1j}}$. Thus so is
	\[
		-(m-1)\sum_{j \in H} x_j \dbreaker{G}{1j} + (n-1)\sum_{j \in H'} x_j \dbreaker{G}{1j} = (n+m-2) \dkirk{H}{1}\breaker{H'}.
	\]
	
	We also have that
	\allstar{
		-(m-1)\sum_{j \in H} x_j \dkirk{G}{1j} + (n-1)\sum_{j \in H'} x_j \dkirk{G}{1j} &= -(m-1)(n-2) \dkirk{H}{1} \kirk{H'} + (n-1)(m-2) \dkirk{H}{1} \kirk{H'} \\
		&= (m-1)\dkirk{H}{1}\kirk{H'}.
	}
	\bigskip
	
	For $j \in H$ pick $A_j,B$ as in the statement of $S(H,e_1)$, this time suppressing $C=0$. Then for
	\allstar{
		A_j' &:= \piecewise{A_j - \frac{(m-1)B}{n+m-2}x_j}{j \in H}{\frac{(n-1)B}{n+m-2}x_j} \\
		B' &:= \bra{1-\frac{m-1}{n+m-2}}B
	}
	we can verify that we have $S(H \bolt H',e)$ in exactly the same way as before.
	
\end{proof}

\begin{figure}[h]
	\centering
	\includegraphics{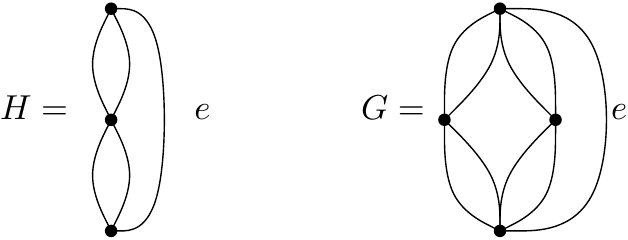}
	\caption{Graphs illustrating the gap between condition 1 and $S$.}
	\label{counterexample}
\end{figure}

Simultaneous combination is stronger than condition 1.  For example, let $G$ and $H$ be as in Figure~\ref{counterexample}, then one can check that $1(H, e)$ is true but $S(H, e)$ is false. However, in some sense simultaneous combination classifies when condition 1 is stable under parallel join. In the example of Figure~\ref{counterexample} we see $G$ is an extension of $H$ by a parallel join but $1(G,e)$ is false.  This is the content of Corollary~\ref{stability cor}.

First we need a few observations on factorizations in Kirchhoff polynomials. Recall that a biconnected component of a graph $G$ is a maximal connected subgraph of $G$ that has no cut vertex. The terms block and biconnected component are synonymous.

% % % % % % % % % % %  New addition
%Loops addressed
% % % % % % % % % % %  New addition
\begin{lemma} \label{lem: kirchhoff factors}
	The non-trivial factors of the Kirchhoff polynomial of a loopless graph $G$ correspond to the biconnected components of $G$ which are not isomorphic to $K_2$.
\end{lemma}
%\begin{lemma} \label{lem: kirchhoff factors}
%	Let $G$ be a graph and let $G \in \mathbb{Q}[x_1, \ldots, x_n]$ be its Kirchhoff polynomial. Then $x_e$ divides $G$ if and only if $e$ is a self-loop of $G$. The non-trivial factors of $G$ coprime to each $x_i$ correspond to the biconnected components of $G$ that are not isomorphic to $K_2$. 
%\end{lemma}

Note that a tree has only copies of $K_2$ as biconnected components corresponding correctly to the Kirchhoff polynomial being $1$.  Also, the result can easily be extended to graphs with self-loops by noting that a self-loop contributes a factor of its variable and any such factors arise in this manner.

\begin{proof}
%	If $e$ is a self-loop of $G$ then it is avoided in every spanning tree of $G$, so $x_e$ divides $G$. Conversely, if $x_e$ divides $G$ then $x_e$ appears in every monomial of $G$, so $e$ must be avoided in every spanning tree of $G$. The only edges avoided in every spanning tree of a graph are self-loops. We now proceed to prove the second statement of the lemma.
	
%	We may assume from this point onward that $G$ does not have self-loops. 
First, note that spanning trees in the different biconnected components meet only at cut vertices, so they are independent.  Thus the Kirchhoff polynomial of $G$ is the product of the Kirchhoff polynomials of these components.
	
	Next we need to show that there are no other factors.  Suppose $G$ has no cut vertices but assume for a contradiction that $\kirk{G} = P \cdot Q$ for some non-constant $P$ and $Q$. Note, a priori, that we do not know if $P,Q$ are Kirchhoff polynomials of a graph. The Kirchhoff polynomial $\kirk{G}$ is linear in each of the edge variables, so any factorization gives a partition of the edges of $G$. Call the two subgraphs the factorization into $P\cdot Q$ induces $A$, whose edges are red, and $B$, whose edges are blue. Let $H^0(X)$ denote the set of connected components of the graph $X$. Let $\Upsilon$ be the graph whose vertices are elements of $H^0(A) \cup H^0(B)$ and whose edges are shared vertices in $G$.
	
	Choose spanning trees $T_{a}$ for each $a \in H^0(A)$ and notice $\bigcup_a T_{a}$ can always be extended to a spanning tree of $G$ using (necessarily) blue edges. Thus 
	\[
	\lambda \bra{\displaystyle \prod_{a \in H^0(A)} \prod_{e \not \in T_a} e}
	\]
	is a monomial of $P$ and $\lambda \neq 0$ since otherwise this product will never show up in the polynomial $\kirk{G}$ contradicting that the spanning trees $T_{a}$ can be extended to a spanning tree of $G$. Similarly, 
	\[
	\mu \bra{\displaystyle \prod_{b \in H^0(B)} \prod_{e \not \in T_b} e}
	\] 
	is a non-zero monomial of $Q$. But now
	\[
	\lambda \mu \bra{\prod_{a \in H^0(A)} \prod_{e \not \in T_a} e}\bra{\prod_{b \in H^0(B)} \prod_{e \not \in T_b} e}
	\]
	is a monomial of $\kirk{G}$. In particular, $\bigcup_a T_{a} \cup \bigcup_b T_{b}$ must be a spanning tree, hence acyclic, so $\Upsilon$ must be a tree. In particular $G$ has at least one cut vertex.
\end{proof}

% % % % % % % % % % %  New addition

\begin{cor} \label{cor: biconnected factors}
	If $G$ is biconnected, does not have self-loops, and not $K_2$, then the Kirchhoff polynomial of $G$ is irreducible and non-constant in every edge variable.
\end{cor}

\begin{proof}
	By Lemma \ref{lem: kirchhoff factors} the irreducibility is immediate. If $G$ is biconnected and not $K_2$ then it does not have a bridge, so each edge is avoided in at least one spanning tree. 
\end{proof}

\begin{cor}
	Let $H$ be a connected subgraph of a biconnected graph $G$ that does not have self-loops. If $\gcd(H,G) \neq 1$, then $H = G$ as graphs.
\end{cor}

\begin{proof}
	Since $\gcd(H,G) \neq 1$ we have that neither $G$ nor $H$ is a tree. As $H$ is connected its Kirchhoff polynomial is non-zero. Since the Kirchhoff polynomial of $G$ is irreducible, by Corollary \ref{cor: biconnected factors} we see that $G$ and $H$ have the same Kirchhoff polynomials. In particular, $H$ is non-constant in all of the edge variables of $G$, so the subgraph $H$ must contain all of the edges of $G$. 
\end{proof}

We remark that the condition that $H$ be a subgraph of $G$ in the previous corollary is essential. 

% % % % % % % % % % %  New addition

\begin{lemma}
  Let $G$ be a connected graph with terminal vertices $s$ and $t$.  If the polynomials $\kirk{G}$ and $\breaker{G}$ have a nontrivial common factor then either
  $G$ has a biconnected component which is connected to the rest of $G$ at a single vertex and includes neither $s$ nor $t$ except possibly one of them as the cut vertex, or $G$ has a self-loop.
  %there is a cut vertex of $G$ which when cut yields a component that either does not include $s$ or does not include $t$.
\end{lemma}

\begin{proof}
	If $G$ has a self-loop the result is immediate, so we may assume otherwise. Let $\Gamma$ be the graph $G$ with $s$ and $t$ identified and note $\breaker{G}$ is the Kirchhoff polynomial $\Gamma$, which by the $\gcd$ condition is non-trivial. Additionally, note that there is a natural morphism of graphs $\varphi\colon G \ra \Gamma$ induced by identifying $s$ and $t$. Let $v$ be the vertex in $\Gamma$ that is the image of both $s$ and $t$ under $\varphi$. Observe that the restriction $\varphi\colon G \bs \{s,t\} \ra \Gamma \bs v$ is an isomorphism.

        Observe that $\deg(\breaker{G}) = \deg(\kirk{G})+1$ so $\breaker{G}$ and $\kirk{G}$ having a factor in common implies that $\breaker{G}$ has a nontrivial factor.  Furthermore, $\Gamma$ is not $K_2$ and so by Corollary~\ref{cor: biconnected factors} we see that $\Gamma$ has at least one cut vertex.
        
	Let $\{ Y_1, \ldots, Y_n \}$ be the biconnected components of $\Gamma$. If $\Gamma$ has a cut vertex $w$ which is not $v$ then there is a $Y_j$ not containing $v$. Thus $\varphi^{-1}(Y_j)$ is a biconnected component of $G \bs \{s,t\}$. There is at most one neighbour of $v$ in $Y_j$, namely $w$, so it follows that $\varphi^{-1}(Y_j)$ is a biconnected component of $G$ which does not contain either $s$ or $t$.

	Otherwise, the unique cut vertex of $\Gamma$ is $v$.
        Suppose $G$ has no biconnected component as described in the statement.  Then either $G$ is biconnected or $G$ is a series join of at least two biconnected components running from $s$ to $t$.  In this latter case, $\Gamma$ is a cycle of biconnected components and hence has no cut vertex, giving a contradiction.
        Finally, suppose the graph $G$ is biconnected.  Since $\gcd(\kirk{G}, \kirk{\Gamma}) \neq 1$, we see that $G$ is not $K_2$ and so by Corollary~\ref{cor: biconnected factors} the polynomial $\kirk{G}$ is irreducible.  Thus by looking at the degrees we see $\breaker{G}=G\cdot L$ where $\deg L = 1$.  Also, by Corollary~\ref{cor: biconnected factors} the polynomial $\kirk{G}$ is non-constant in every edge variable so any variable in $L$ appears quadratically in $\breaker{G}$, which is a contradiction.  Thus $G$ has a biconnected component as described in the statement.
     	
\end{proof}

%\begin{proof}
%	Let $\Gamma$ be the graph $G$ with $s$ and $t$ identified and note $\breaker{G}$ is the Kirchhoff polynomial $\Gamma$.  By the previous lemma, $G$ and $\Gamma$ have a common biconnected component.  Thus there is a common cut vertex of $G$ and of $\Gamma$, hence a cut vertex of $G$ which has a component that does not include one of $s$ or $t$.
%\end{proof}

% % % % % % % % % % %
% % D R A W I N G % %
% % % % % % % % % % %

\begin{figure}[h]
	\centering
	\includegraphics[page=1,width=.45\textwidth]{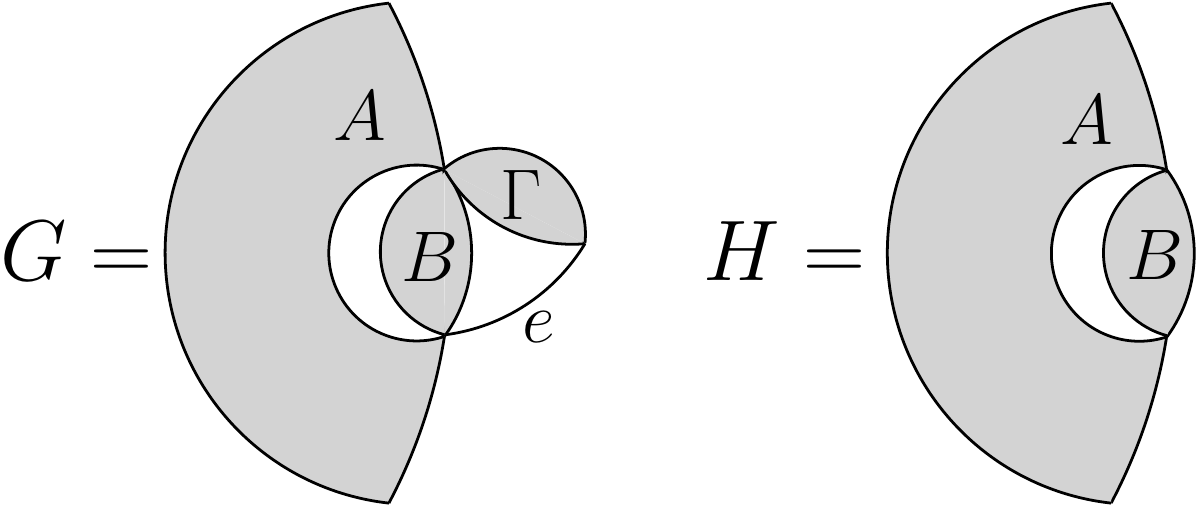}
	\caption{Graphs used for proof of Proposition \ref{stability}}
	\label{fig1}
\end{figure}

\begin{prop} \label{stability}
	Let $G$ and $H$ be series-parallel graphs of the form illustrated in Figure \ref{fig1} and suppose condition 1 holds for $e$ in $G\paral H$. Then $S(G,e)$ holds.
\end{prop}

\begin{proof}
	Let $e=e_1$ and note that series-parallel graphs do not contain self-loops.
	
	In what follows we identify the edge variables of $H$ with their counterparts in $H \inj G$. Let $\Gamma := G \bs (H \cup \set{e})$. 
	Note $\dkirk{G}{1} = \kirk{H}\kirk{\Gamma}$ and $\dbreaker{G}{1} = \breaker{H}\kirk{\Gamma}$.
	Since condition 1 holds we may write $G \paral H$ as
	\allstar{
		\kirk{G}\breaker{H} + \breaker{G}\kirk{H} = & 
		\sum_{j \in H \subset G} B'_j \dkirk{H}{j} \kirk{\Gamma}\breaker{H} + \sum_{j\in \Gamma} C_j\kirk{H}\dkirk{\Gamma}{j}\breaker{H} + \sum_{j \in H}B''_j\kirk{H}\kirk{\Gamma}\dbreaker{H}{j} \\
		& + \sum_{j \in H \subset G} B'_j \dbreaker{H}{j} \kirk{\Gamma}\kirk{H} + \sum_{j\in \Gamma} C_j\breaker{H}\dkirk{\Gamma}{j}\kirk{H} + \sum_{j \in H}B''_j\breaker{H}\kirk{\Gamma}\dkirk{H}{j} + D\kirk{H}\kirk{\Gamma}\breaker{H}\\
		= & \breaker{H} \bra{ \sum_{j \in H} B_j \dkirk{G}{1j} + \sum_{j \in \Gamma} C_j \dkirk{G}{1j} } 
		+ \kirk{H} \bra { \sum_{j \in H} B_j \dbreaker{G}{1j}  + \sum_{j \in \Gamma} C_j \dbreaker{G}{1j} } + D\kirk{H}\kirk{\Gamma}\breaker{H} \\
	}
	where $B_j = B_j' + B_j''$.
	Since $H$ is series-parallel it does not have a cut vertex with a biconnected component joining only at that vertex and containing neither $s$ nor $t$ except possibly as the cut vertex, so $\gcd(\kirk{H},\breaker{H}) = 1$. Thus from rearranging and factoring we get
	\begin{align}
		\breaker{H} &\mid \breaker{G} - \bra { \sum_{j \in H} B_j \dbreaker{G}{1j}  + \sum_{j \in \Gamma} C_j \dbreaker{G}{1j}} \label{div eq 1}\\
		\kirk{H} &\mid \kirk{G} - \bra{ \sum_{j \in H} B_j \dkirk{G}{1j} + \sum_{j \in \Gamma} C_j \dkirk{G}{1j} } \label{div eq 2}.
	\end{align}
	If $\kirk{\Gamma}$ were also to divide each of the expressions above, then by disjointness of the variables of $H,\Gamma$ we see $\dbreaker{G}{1} = \breaker{H}\kirk{\Gamma}$ divides the first expression and $\dkirk{G}{1} = \kirk{H}\kirk{\Gamma}$ divides the second. Whence, $S(G,e)$ would hold. We proceed to prove exactly this.
	
	Notice for $j \in H$ that $\Gamma$ remains a 1-cut component of $G \bs \set{1,j}$, so $\kirk{\Gamma} \mid \dbreaker{G}{1j}$ and $\kirk{\Gamma} \mid \dkirk{G}{1j}$. Since $H$ is connected we get by the 2-vertex-cut formula that
	\allstar{
		\kirk{G} &= \kirk{H}(x_e\kirk{\Gamma} + \breaker{\Gamma}) + \overset{uv}{H}\kirk{\Gamma} \\
		\breaker{G} &= \breaker{H}(x_e\kirk{\Gamma} + \breaker{\Gamma})  \overset{st, uv}{H}\kirk{\Gamma}.
	} 
	where $u,v$ are the source and terminal of $\Gamma \bolt e$. So the right hand sides of \eqref{div eq 1} and \eqref{div eq 2} modulo $\gen{\kirk{\Gamma}}$ become
	\allstar{
		\kirk{H}\breaker{\Gamma} - \sum_{j \in \Gamma} C_j \dkirk{G}{1j} + \gen{\kirk{\Gamma}} \\
		\breaker{H}\,\breaker{\Gamma} - \sum_{j \in \Gamma} C_j \dbreaker{G}{1j} + \gen{\kirk{\Gamma}}. \\
	} 
	Or further simplifying with our explicit expressions of $\breaker{G},\kirk{G}$ we get
	\allstar{
		\kirk{H}\breaker{\Gamma} - \sum_{j \in \Gamma} C_j \kirk{H}\dkirk{\Gamma}{j} + \gen{\kirk{\Gamma}} \\
		\breaker{H}\,\breaker{\Gamma} - \sum_{j \in \Gamma} C_j \breaker{H}\dkirk{\Gamma}{j} + \gen{\kirk{\Gamma}}. \\
	}
	Again, since $\kirk{H},\breaker{H}$ are coprime to $\kirk{\Gamma}$ it suffices to verify 
	\[
	\kirk{\Gamma} | \breaker{\Gamma} - \sum_{j \in \Gamma} C_j \dkirk{\Gamma}{j}. 
	\]
	Returning to the original expression we see that
	\allstar{
		\kirk{G\paral H} + \gen{\kirk{\Gamma}}= & \sum_{j \in H} B_j \dkirk{G}{1j} \breaker{H} + \sum_{j \in \Gamma} C_j \dkirk{G}{1j} \breaker{H}  \\
		& +   \sum_{j \in H} B_j \dbreaker{G}{1j} \kirk{H} + \sum_{j \in \Gamma} C_j \dbreaker{G}{1j} \kirk{H} + \gen{\kirk{\Gamma}} \\
		= & \sum_{j \in \Gamma} C_j \dkirk{G}{1j} \breaker{H} + \sum_{j \in \Gamma} C_j \dbreaker{G}{1j} \kirk{H} + \gen{\kirk{\Gamma}} \\
		= & \sum_{j \in \Gamma} C_j \dkirk{\Gamma}{j} \kirk{H} \breaker{H} + \sum_{j \in \Gamma} C_j \dkirk{\Gamma}{j}\breaker{H} \kirk{H} + \gen{\kirk{\Gamma}} \\
		= & 2\sum_{j \in \Gamma} C_j \dkirk{\Gamma}{j} \kirk{H} \breaker{H} + \gen{\kirk{\Gamma}}.
	}
	Meanwhile, from Figure \ref{fig1} we see by the two vertex cut formula that
	\allstar{
		\kirk{G\paral H} + \gen{\kirk{\Gamma}} &= (\breaker{\Gamma} + x_e \kirk{\Gamma}) (\kirk{H\paral H}) + \kirk{\Gamma} (\kirk{H}\overset{st,uv}{H} + \breaker{H}\overset{uv}{H}) + \gen{\kirk{\Gamma}} \\
		&= \breaker{\Gamma}(\kirk{H\paral H}) + \gen{\kirk{\Gamma}} \\
		&= 2\breaker{H}\kirk{H}\breaker{\Gamma} + \gen{\kirk{\Gamma}}.
	}
	Thus taking the difference of the two previous calculations we obtain
	\[
		0 = 2\breaker{H}\kirk{H}\bra{\breaker{\Gamma} - \sum_{j \in \Gamma} C_j \dkirk{\Gamma}{j}} + \gen{\kirk{\Gamma}}.
	\]
	Since both $\breaker{H}$ and $\kirk{H}$ are coprime to $\kirk{\Gamma}$, we may cancel these terms to obtain the result. 
\end{proof}

\begin{cor}\label{stability cor}
	Suppose $G$ is a series-parallel graph such that $S(G,e)$ does not hold but Condition 1 does hold for $e$.  Then there is a series-parallel graph $H$ such that Condition 1 fails for $e$ in $G\paral H$.
\end{cor}

\begin{proof}
	Note that $G$ is series-parallel and so does not have self-loops. Condition 1 holds for $e$ in $G$, so $e$ is regular and so at some point in the construction of $G$ by series and parallel operations we have the subgraph $B \paral (C_1\bolt e \bolt C_2)$, where one or both of $C_1$ and $C_2$ may be empty.  Without loss of generality we may instead consider $B \paral (C\bolt e)$ where $C = C_1\bolt C_2$ since the Kirchhoff polynomial is invariant under this transformation. By Proposition \ref{2edge} we may subdivide $e$ if necessary to ensure that $C$ is non-empty (if $e$ is subdivided the two halves are the 2-edge cut set). Therefore $G$ is of the form to apply Proposition~\ref{stability}.  Let $H$ be as in Proposition~\ref{stability}.  Since $S(G,e)$ does not hold by assumption, we see by Proposition~\ref{stability} that condition 1 does not hold for $e$ in $G\paral H$.
\end{proof}

%\begin{proof}
%	Suppose to the contrary that Condition 1 held for all such $G\paral H$.

%	Note that $G$ is series-parallel and so does not have self-loops. Condition 1 holds for $e$ in $G$, so $e$ is regular and so at some point in the construction of $G$ by series and parallel operations we have the subgraph $B \paral (C_1\bolt e \bolt C_2)$ where one or both of $C_1$ and $C_2$ may be empty.  Without loss of generality we may instead consider $B \paral (C\bolt e)$ where $C = C_1\bolt C_2$ since the Kirchhoff polynomial is invariant under this transformation. By Proposition \ref{2edge} we may subdivide $e$ if necessary to ensure that $C$ is non-empty (if $e$ is subdivided the two halves are the 2-edge cut set). Therefore $G$ is of the form to apply Proposition~\ref{stability}.  Let $H$ be as in Proposition~\ref{stability}.  By assumption condition 1 holds for $e$ in $G\paral H$ and so by Proposition~\ref{stability} $S(G,e)$ holds which is a contradiction.
%\end{proof}

\section{Carving out a series-parallel class for condition 1}

As we saw in the previous section simultaneous combination is well behaved with respect to series and parallel joins, implies condition 1, and is strictly stronger than condition 1.  Furthermore the way in which it is stronger is itself well behaved under parallel join in the sense made precise in Corollary~\ref{stability cor}.

%Consequently, characterizing a class of series-parallel graphs which have the simultaneous combination property would give a class of series-parallel graphs satisfying condition 1.  Unfortunately, simultaneous combination is not well suited to pulling out a good base case for implementing this plan. We need instead a variant on simultaneous combinations defined as follows.  The point of this condition shows up most strongly in Corollary~\ref{T-series-lift-cor}; specifically it captures when a parallel join with edge $e$ will satisfy $S(G,e)$.

Consequently, characterizing a class of series-parallel graphs which have the simultaneous combination property would give a class of series-parallel graphs whose edges all satisfy condition 1. Due to the nice behaviour with respect to series and parallel joins one would expect that we could give a recursively defined family of series parallel graphs with the simultaneous combination property. Unfortunately, simultaneous combination is not well suited to pulling out a good base case for implementing this plan.  We need instead a variant on simultaneous combinations to generate a larger class of graphs that satisfy $S(G)$. The point of this condition shows up most strongly in Corollary~\ref{T-series-lift-cor}; specifically it captures when a parallel join with edge $e$ will satisfy $S(G,e)$. As an added benefit we are then able to apply these technical results to identify a combinatorial condition that ensures $S(G)$.

\begin{definition} \label{simultaneous combination II}
Let $G$ be a series-parallel graph. We say %that \emph{simultaneous combination (type II)} holds for $G$, or 
$T(G)$ holds, if there are polynomials $A_j,C$ such that
\allstar{
\breaker{G} &= \sum A_j \dkirk{G}{j}\\
C \cdot \breaker{G} &= \sum A_j \dbreaker{G}{j} \in \gen{\breaker{G}}.
}
\end{definition}

We give some examples of condition $T$ for small graphs in Figure~\ref{fig:TofG}. It is helpful to compare these examples to the example in Figure~\ref{counterexample}. Condition 1 is false for the edge $e \in G$ in Figure~\ref{counterexample}. By Corollary~\ref{stability cor} we see that $S(H,e)$ does not hold, where $H$ is the graph from Figure~\ref{counterexample}. Finally, we will see by Lemma~\ref{bridge lemma}(i) that condition $T$ fails for the graph in Figure~\ref{fig:TofG}b, explaining why $S(H, e)$ fails. 

Condition $T$ is well behaved with respect to parallel join and we can understand both joins with paths.

% % % % % % % % % % %
% % D R A W I N G % %
% % % % % % % % % % %
\begin{figure}[h]
	\centering
	\includegraphics[width=0.36\textwidth]{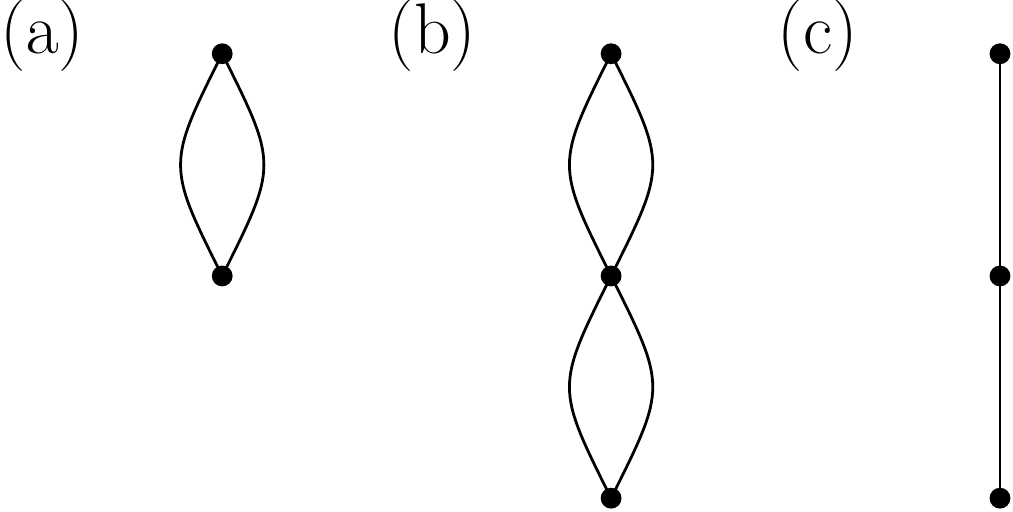}
	\caption{Graph (a) satisfies condition $T$. Graph (b) does not satisfy condition $T$, which demonstrates that condition $T$ is not well behaved with respect to series join. Finally, graph (c) does not satisfy condition $T$ for the trivial reason that the derivatives of its Kirchhoff polynomial are identically 0.}
	\label{fig:TofG}
\end{figure}

\begin{lemma} \label{T-parallel-lift}
If $H,H'$ are series-parallel graphs and $T(H)$ then $T(H\paral H')$.
\end{lemma}

\begin{proof}
Let $G := H\paral H'$ and note that $\kirk{G} = \kirk{H}\breaker{H'} + \breaker{H} \kirk{H'}$. Let $n := \deg{\breaker{H}}$ and let $m := \deg(\breaker{H'})$. Thus
\[
\gen{\dkirk{G}{j}} = \gen{ \dkirk{H}{j} \breaker{H'} + \dbreaker{H}{j} \kirk{H'}, \kirk{H}\dbreaker{H'}{j} + \breaker{H} \dkirk{H'}{j}}.
\]
Now
\[
\breaker{H}\kirk{H'} = -\frac{1}{n+m-1}\bra{(m-1)\sum_{j \in H} x_j\dkirk{G}{j} - n\sum_{j \in H'} x_j \dkirk{G}{j}}.
\]
By $T(H)$ we may choose $B_j, j \in H$ and $C$ such that
\[
\sum_{j \in H} B_j \dkirk{G}{j} = \sum_{j \in H} B_j ( \dkirk{H}{j} \breaker{H'} + \dbreaker{H}{j} \kirk{H'}) = \breaker{H}\,\breaker{H'} + C\breaker{H}\kirk{H'}.
\]
Let
\[
B_j' := \piecewise{B_j - \frac{C(m-1)}{n+m-1}x_j}{j \in H}{\frac{nC}{n+m-1}x_j}
\]
We verify that $T(G)$ is satisfied with a calculation. 
\allstar{
\sum_{j \in G} B_j' \dkirk{G}{j} &= \sum_{j \in H} \bra{B_j - \frac{C(m-1)}{n+m-1}x_j} \dkirk{G}{j} + \sum_{j \in H'} \bra{\frac{nC}{n+m-1}x_j} \dkirk{G}{j} \\
&= \breaker{H}\,\breaker{H'} + C\breaker{H}\kirk{H'} - C\breaker{H}\kirk{H'} \\
&= \breaker{G}.
}
\noindent
Meanwhile
\allstar{
\sum_{j \in G} B_j' \dbreaker{G}{j} &=  \sum_{j \in H} \bra{B_j - \frac{C(m-1)}{n+m-1}x_j} \dbreaker{H}{j}\breaker{H'} + \sum_{j \in H'} \bra{\frac{nC}{n+m-1}x_j} \breaker{H}\,\dbreaker{H'}{j} \\
&= \breaker{H'} \bra{\sum_{j \in H} B_j \dbreaker{H}{j}} +   \frac{C}{n+m-1} \bra{\sum_{j \in H'} nx_j \breaker{H}\, \dbreaker{H'}{j}  - \sum_{j \in H} (m-1)x_j \dbreaker{H}{j}\breaker{H'}} \\
&= \breaker{H'} \bra{\sum_{j \in H} B_j \dbreaker{H}{j}} +  \frac{C}{n+m-1}  \bra{nm \breaker{H}\, \breaker{H'}  -  n(m-1) \breaker{H}\,\breaker{H'}}
}
\noindent
which by $T(H)$ is in the ideal $\gen{\breaker{H}\,\breaker{H'}}$.
\end{proof}

Condition $T$ has special behaviour for paths which we illustrate with the following lemmas.

\begin{lemma} \label{T-path-lift}
Let $H$ be a series-parallel graph such that $T(H)$. Then $T(H \bolt K_2)$.
\end{lemma}

\begin{proof}
Let $e$ be the edge of $K_2$, let $n := \deg \kirk{H}$, and let $A_j,C$ be polynomials such that
\allstar{
\sum_{j \in H} A_j \dkirk{H}{j} &= \breaker{H} \\
\sum_{j \in H} A_j \dbreaker{H}{j} &= C \cdot \breaker{H}.
}
Note that for the polynomials $\kirk{H \bolt K_2} = \kirk{H}$ and $\breaker{H \bolt K_2} = \breaker{H} + x_e \kirk{H}$. Now let 
\[
C' := \bra{C + x_e + \frac{n+1}{n}x_e}
\]
and observe
\[
\sum_{j \in H} \bra{A_j + \frac{x_e}{n} x_j} \dkirk{H}{j}  +  x_e \bra{C' - x_e} \frac{\der}{\der x_e} (\kirk{H}) = \breaker{H} + x_e \kirk{H} + 0
\]
together with
\allstar{
& \sum_{j \in H} \bra{A_j + \frac{x_e}{n} x_j} \bra{\breaker{H} + x_e \kirk{H}}^j +  x_e \bra{C' - x_e} \frac{\der}{\der x_e} \bra{\breaker{H} + x_e \kirk{H}} \\
=& \ \  C \cdot \breaker{H} + x_e \breaker{H} + \frac{n+1}{n} x_e \breaker{H} + x_e^2 \kirk{H} + x_e \bra{C' - x_e} \kirk{H} \\
=& \bra{C + x_e + \frac{n+1}{n}x_e} \breaker{H} + C' x_e \kirk{H} \\
=& \ \ C' (\breaker{H} + x_e\kirk{H})
}
proves the result.
\end{proof}

\begin{lemma} \label{path-to-T}
Let $H$ be a series-parallel graph and $\Gamma$ a path. Then $T(H\paral \Gamma)$.
\end{lemma}

\begin{proof}
Let $G := H\paral \Gamma$, let $n := \deg(\kirk{G})$, and let $z_1, \ldots, z_m$ be the edge variables for $\Gamma$. Now
\[
\breaker{G} := (z_1 + \ldots + z_m) \breaker{P}, \midskip \kirk{G} = (z_1 + \ldots + z_m)\kirk{H} + \breaker{H}.
\]
From which it is clear that
\begin{align*}
	\breaker{G} = \ & (z_1 + \ldots + z_m)\bra{\frac{1}{n} \sum_{j \in G} x_j \dkirk{G}{j} - \sum_{j \in \Gamma} z_j \dkirk{G}{j}} 
	\intertext{ and }
	  & (z_1 + \ldots + z_m)\bra{\frac{1}{n}\sum_{j \in G} x_j \dbreaker{G}{j} - \sum_{j \in \Gamma} z_j \dbreaker{G}{j}} \\
	= \ & (z_1 + \ldots + z_m)\bra{\frac{n+1}{n} \breaker{G} - \breaker{G}} \\
	\in \ & \gen{\dbreaker{G}{j}} ,
\end{align*}
so we have $T(G)$.
\end{proof}

Now we are positioned to see what condition $T(H)$ is really for, namely guaranteeing $S(G,e)$ when $G$ is $H$ parallel joined with $e$, and more generally for series joins of such $H$s.

\begin{lemma} \label{T-series-lift}
Let $\mathcal{H}$ be a finite collection of series-parallel graphs such that $T(H)$ holds for each $H \in \mathcal{H}$. Let $G$ be the series join of all $H \in \mathcal{H}$. Then $\breaker{G} \in \gen{\dkirk{G}{j}}$. 
\end{lemma}

\begin{proof}

We have that
\[
\kirk{G} = \prod_{H \in \mathcal{H}} \kirk{H}, \midskip \breaker{G} = \sum_{H \in \mathcal{H}} \breaker{H} \prod_{H' \neq H} \kirk{H'}.
\]
So
\[
\gen{\dkirk{G}{j}} = \bigoplus_{H \in \mathcal{H}} \gen{\dkirk{H}{j} \prod_{H' \neq H} \kirk{H'} : j \in H}.
\]
By $T(H_i)$ we may choose $B_j^{H}$ such that
\[
\sum_{H \in \mathcal{H}} \sum_{j \in H} B_j^{H} \dkirk{G}{j} =\sum_{H \in \mathcal{H}} \breaker{H} \prod_{H' \neq H} \kirk{H'}.
\]
\end{proof}

The next corollary is an easy consequence of Lemma \ref{T-series-lift} but is not essential to the rest of the paper. 

\begin{cor} \label{T-series-lift-cor}
Let $\mathcal{H}$ be a finite collection of series-parallel graphs such that $T(H)$ holds for each $H \in \mathcal{H}$. Let $G$ be the series join of all $H \in \mathcal{H}$.  Then Condition 1 holds for (the edge of) $K_2$ in $K_2\paral G$.
\end{cor}

\begin{proof}
Let $\Gamma := K_2\paral G$ and $e$ be the edge of $K_2$. Then
\[
\kirk{\Gamma} := x_e\kirk{G} + \breaker{G}.
\]
By the previous lemma Condition 1 holds for $e$ in $\Gamma$.
\end{proof}

Next we look at the base cases for building series-parallel graphs with the simultaneous combination property. A notable obstacle in classifying series-parallel graphs satisfying simultaneous combination is the fact that a series parallel graph $G$ can decompose as $H \paral H'$ where one of $H$ or $H'$ has edges which are not regular. Lemma \ref{bridge lemma} shows how we can sometimes overcome this obstacle. The distinct cases of Lemma \ref{bridge lemma} are due to the fact that condition $T$ and simultaneous combination are sensitive to the marking of source and terminal. 

\begin{lemma} \label{bridge lemma}
Let $H := H_1 \bolt \ldots \bolt e_1 \bolt \ldots \bolt H_r$ be a series-parallel graph where $e_1 \in H$ is a bridge and let $\Gamma$ be a series-parallel graph. If any of 
\begin{enumerate}[(i)]
\item
$T(\Gamma)$ holds and $H$ is a path,
\item
$T(\Gamma)$ holds and $T(H_i)$ holds for all $i$,
\item
$\Gamma$ is a path and $T(H_i)$ holds for all $i$
\end{enumerate}
are satisfied, then $S(H\paral \Gamma,e_1)$ and $T(H\paral \Gamma)$.

\end{lemma}

\begin{proof}
Note $T(\Gamma)$ implies $T(\Gamma \paral H)$ so this condition is satisfied in cases $(i)$ and $(ii)$. Let $G := H\paral \Gamma$. We prove $S(G,e_1)$ holds case by case. 

\begin{enumerate}[(i)]
\item
Note that
\[
\kirk{G} = ( x_1 + \ldots + x_n)\kirk{\Gamma} + \breaker{\Gamma}, \midskip \breaker{G} = ( x_1 + \ldots + x_n) \breaker{\Gamma}.
\]
Let $A_j,C$ be the polynomials as in $T(\Gamma)$. Then
\[
\sum_{j \in \Gamma} A_j \dkirk{G}{1j} + ( x_1 + \ldots + x_n)\dkirk{G}{1}= \kirk{G}
\]
and
\[
\sum_{j \in \Gamma} A_j \dbreaker{G}{1j} + (( x_1 + \ldots + x_n)-C)\dbreaker{G}{1} = \breaker{G}
\]
so we are done.

\item
We observe the claim is equivalent to proving $S( (H_1 \bolt \ldots \bolt H_r \bolt e_1)\paral \Gamma,e_1)$ since the relevant polynomials associated to these two graphs are identical. Let us relabel $H := H_1 \bolt \ldots \bolt H_r$ and $G := (H \bolt e_1) \paral \Gamma$. Now
\allstar{
\kirk{G} &= \kirk{\Gamma}(\kirk{H}x + \breaker{H}) + \kirk{H}\breaker{\Gamma} \\
\breaker{G} &= \breaker{\Gamma}(\kirk{H} x + \breaker{H}).
}

Choose $B_j,D$ as in $T(\Gamma)$. Since $T(H_i)$ for all $i$ we have by Lemma \ref{T-series-lift} that there are $C_j$ such that 
\[
\sum_{j \in H} C_j \dkirk{H}{j} = \breaker{H}.
\]
Letting $m := \deg \Gamma$ we verify that
\[
D' := -\bra{D + \frac{x}{m}}, \midskip A_j := \piecewise{\frac{x x_j}{m} + B_j}{j \in \Gamma}{C_j}
\]
are certificates for  $S(H\paral\Gamma,x)$:
\allstar{
\sum_{j \in G} A_j \dkirk{G}{1j} &= \sum_{j \in H} C_j \kirk{\Gamma}\dkirk{H}{j} + \sum_{j \in \Gamma} \bra{\frac{x x_j}{m} + B_j} \dkirk{\Gamma}{j}{\kirk{H}} \\
&= \kirk{\Gamma}\breaker{H} + x\kirk{\Gamma}\kirk{H} + \breaker{\Gamma} \kirk{H}.
}
Meanwhile
\allstar{
	\sum_{j \in G} A_j \dbreaker{G}{1j} - D \dbreaker{G}{1} &= \sum_{j \in H} C_j \breaker{\Gamma}\dkirk{H}{j} + \sum_{j \in \Gamma} \bra{\frac{x x_j}{m} + B_j} \dbreaker{\Gamma}{j}{\kirk{H}} - \bra{D + \frac{x}{m}} \breaker{\Gamma} \kirk{H}\\
	&= \breaker{\Gamma} \, \breaker{H} + x\breaker{\Gamma}\kirk{H}.
}

\item
%Again we observe the claim is equivalent to proving $S( (H_1 \bolt \ldots \bolt H_r \bolt e_1)\paral \Gamma,e_1) = S((\hat H \bolt e_1)\paral \Gamma, e_1)$. (Relabel $\hat H$ to be $H$). Write $\breaker{\Gamma} = z_1 + \ldots + z_m$ and note
Again we observe the claim is equivalent to proving $S( (H_1 \bolt \ldots \bolt H_r \bolt e_1)\paral \Gamma,e_1)$, so we relabel $H := H_1 \bolt \ldots \bolt H_r$.  Write $\breaker{\Gamma} = z_1 + \ldots + z_m$ and note
\allstar{
\kirk{G} &= \kirk{H}(z_1 + \ldots + z_m) + (x_e \kirk{H} + \breaker{H}) \\
\breaker{G} &= (x_e \kirk{H} + \breaker{H})(z_1 + \ldots + z_m).
}
Since $T(H_i)$ holds for all $i$ we may choose $A_j$ as in Lemma \ref{T-series-lift}. Now let
\allstar{
B_j &:= \piecewise{A_j}{j \in H}{0} \\
C    &:= (z_1 + \ldots + z_m + x_e), \midskip D := x_e.
}
Then
\allstar{
\sum_{j \in G} B_j \dkirk{G}{ej} + C\dkirk{G}{e} &= \breaker{H} + (z_1 + \ldots + z_m + x_e)\kirk{H} \\
\sum_{j \in G} B_j \dbreaker{G}{ej} + D \dbreaker{G}{e} &= (z_1 + \ldots + z_m)\breaker{H} + x_e(z_1 + \ldots + z_m)\kirk{H}  
}
and so $S(G,e_1)$ holds. In this case we also need to show $T(G)$, but this follows from Lemma \ref{path-to-T}. 
\end{enumerate}
\end{proof}

Cycles are graphs which behave very nicely under series and parallel joins despite the fact that condition 1 is false for every edge. It will be convenient to use cycles as building blocks for larger graphs satisfying simultaneous combination.

\begin{lemma} \label{cycle lemma}
Let $H$ be a series-parallel graph which is a cycle and $\Gamma$ be any series-parallel graph. Then $T(H)$, $T(H \paral \Gamma)$, and $S(H\paral \Gamma, e)$ holds for any edge $e\in H$. If $\Gamma$ is not a path then $S(H \bolt \Gamma, e)$ as well.
\end{lemma}

\begin{proof}
Write $H := \P \P'$ with $\P, \P'$ paths. Let $\Gamma$ be a series-parallel graph. By Lemma \ref{path-to-T} we get that all of $T(\P\paral \Gamma)$, $T(\P'\paral  \Gamma)$, and $T(H)$ hold. Thus if $e \in \P$ we have $S(\P \paral (\Gamma \paral \P'), e)$ by Lemma \ref{bridge lemma}. Moreover, $T(\Gamma\paral  \P)$ implies $T(\Gamma\paral H)$ by Lemma \ref{T-parallel-lift}. 

Finally we consider series join.  Let $G := H \bolt \Gamma$, let $x_i$ denote the edge variables of $\P$, and let $y_i$ denote the edge variables of $\P'$. Note
	\allstar{
		\breaker{G} &= (x_1 + \ldots + x_n)(y_1 + \ldots + y_m)\kirk{\Gamma} + ( x_1 + \ldots + x_n + y_1 + \ldots + y_m)\breaker{\Gamma} \\
		\kirk{G} &= ( x_1 + \ldots + x_n + y_1 + \ldots + y_m)\kirk{\Gamma} \\
		&= \kirk{H} \kirk{\Gamma}.
	}
Then
	\allstar{
		\gen{\dkirk{G}{1j}} &= \gen{ 0, \dkirk{\Gamma}{j}} \\
		\gen{\dbreaker{G}{1j}} &= \gen{0,\kirk{\Gamma}, (y_1 + \ldots + y_m)\dkirk{\Gamma}{j} + \dbreaker{\Gamma}{j}}.
	}
Let $k := \deg{\kirk{\Gamma}}$. We verify that the choice of
	\allstar{
		B_j &:= 
		\begin{cases}
			\frac{1}{k}\kirk{H} x_j & \text{if }j \in \Gamma \\
			\breaker{H} - (y_1 + \ldots + y_m)(1-\frac{1}{k})\kirk{H} & \text{if } x_j = y_1 \\
			0 & \text{otherwise }
		\end{cases} \\
		C &:= -\frac{1}{k}\kirk{H}
	}
works as a certificate for $S(G,e_1)$. We see that
	\allstar{
		\sum_{j \in G} B_j \dkirk{G}{1j} = \sum_{j \in \Gamma} \frac{1}{k}\kirk{H} x_j \dkirk{\Gamma}{j} + 0
	}
and that
	\allstar{
		& \ \ \sum_{j \in G} B_j \dbreaker{G}{1j} + C \dbreaker{G}{1} \\
		&= \sum_{j \in \Gamma} \frac{1}{k}\kirk{H} x_j \bra{(y_1 + \ldots + y_m)\dkirk{\Gamma}{j} + \dbreaker{\Gamma}{j}} + \bra{\breaker{H} - (y_1 + \ldots + y_m)\bra{1-\frac{1}{k}}\kirk{H}} \kirk{\Gamma} \\
		& \ \ - \frac{1}{k}\kirk{H} \bra{(y_1 + \ldots + y_m)\kirk{\Gamma} + \breaker{\Gamma}} \\ 
		&= (y_1 + \ldots + y_m) \kirk{H} \kirk{\Gamma} + \frac{k+1}{k}\kirk{H}\breaker{\Gamma} + \breaker{H} \kirk{\Gamma} - (y_1 + \ldots + y_m)\bra{1-\frac{1}{k}}\kirk{H} \kirk\Gamma \\
		& \ \ - \frac{1}{k}\kirk{H} \bra{(y_1 + \ldots + y_m)\kirk{\Gamma} + \breaker{\Gamma}} \\ 
		&=  \kirk{H}\breaker{\Gamma} + \breaker{H} \kirk{\Gamma} + \bra{1 - \bra{1-\frac{1}{k}} - \frac{1}{k}}(y_1 + \ldots + y_m)\kirk{H}\kirk{\Gamma}.
	}
So we are done.
\end{proof}

\begin{cor} \label{cyclepaths}
Let $H$ be a cycle, let $\P$ be a path, let $M := H \bolt \P$, and let $\Gamma$ be a series-parallel graph. Then
	\begin{enumerate}[(i)]
		\item
		$T(M)$,
		\item
		$T(M \paral \Gamma)$,
		\item
		$S(M \paral \Gamma, e)$ for all $e \in H$, and
		\item
		if $T(\Gamma)$ then $S(M \paral \Gamma,e)$ for all $e \in M$. 
	\end{enumerate}
\end{cor}

\begin{proof}
We obtain $T(M)$ by repeatedly invoking Lemma \ref{T-path-lift} and $T(M \paral \Gamma)$ from Lemma \ref{T-parallel-lift}. Denote $G := M \paral \Gamma$. If $e \in \P$ and $T(\Gamma)$ then we have $S(G,e)$ from Lemma \ref{bridge lemma}. Therefore (iv) follows from (iii) and all that is left to prove is (iii). If $e \in H$, let $n := \deg(\breaker{\Gamma})$ and write
\begin{gather*}
\breaker{H} = (e + x_1 + \ldots + x_r)(y_1 + \ldots + y_m) \\
 \kirk{H} =  e + x_1 + \ldots + x_r + y_1 + \ldots + y_m \\
\breaker{\P} = (z_1 + \ldots + z_k), \midskip \kirk{\P} = 1.
\end{gather*}
Then
\allstar{
\kirk{G} &= \kirk{H}\breaker{\Gamma} + \bra{\breaker{H} +\kirk{H} (z_1 + \ldots + z_k)}\kirk{\Gamma} \\
\breaker{G} &= \breaker{M}\,\breaker{\Gamma} = \bra{\breaker{H} +\kirk{H} (z_1 + \ldots + z_k)}\breaker{\Gamma} \\
\dkirk{G}{e} &= \breaker{\Gamma} + \bra{y_1 + \ldots + y_m + z_1 + \ldots + z_k}\kirk{\Gamma} \\
\dbreaker{G}{e} &=  \bra{y_1 + \ldots + y_m + z_1 + \ldots + z_k}\breaker{\Gamma}.
}

We notice $\dkirk{G}{ey_1} = \kirk{\Gamma}$, $\dbreaker{G}{ey_1} = \breaker{\Gamma}$, that
\begin{align*}
	& \ \ \ \  \sum_{j \in \Gamma} x_j \kirk{H} \dkirk{G}{ej} - (n-1)\kirk{H}\dkirk{G}{e} \\
	&=  n\kirk{H}\breaker{\Gamma} + (n-1)\kirk{H} \bra{y_1 + \ldots + y_m + z_1 + \ldots + z_k}\kirk{\Gamma}  - (n-1)\kirk{H}\dkirk{G}{e} \\
	&= \kirk{H}\breaker{\Gamma}, 
	\intertext{and that}
	& \ \ \ \  \sum_{j \in \Gamma} x_j \kirk{H} \dbreaker{G}{ej} \\
	&=  n\kirk{H}\bra{y_1 + \ldots + y_m + z_1 + \ldots + z_k}\breaker{\Gamma}\\
	&= n\kirk{H}\dbreaker{G}{e}.
\end{align*}
So it follows
\begin{gather*}
A_j := \piecewiseThree{x_j\kirk{H}}{j \in \Gamma}{\breaker{H\bolt\P}}{j = y_1}{0} \\
B := -(n-1)\kirk{H}, \midskip C := -n\kirk{H}
\end{gather*}
serve as witnesses for $S(G,e)$.
\end{proof}

We are now ready to give a class of series-parallel graphs which satisfy condition 1.  A key class of graphs is those which have a planar embedding where one of the faces gives a Hamiltonian cycle, i.e. the cycle defined by the face includes all the vertices of the graph.

For convenience in what follows we make the following definition.

\begin{definition}
We call the operation $\paral^e(G) := G \paral e$ the \emph{restricted parallel join}. Similarly, we call $\bolt^e(G) := e \bolt G$ and $\bolt_e(G) := G \bolt e$  the \emph{restricted series joins}.
\end{definition}

\begin{remark} \label{ham arc-diagram}
Let $G$ be a graph with a planar embedding that has a Hamiltonian cycle as a facial cycle and let $s,t$ be two vertices which are consecutive on this cycle (joined by the edge $e$). We may always view $G$ as a non-crossing arc diagram with the Hamilton path from $s$ to $t$ as a horizontal line segment whose left endpoint is $s$ and right endpoint $t$ and all other edges as arcs above this line segment. Since the aforementioned Hamilton path together with the edge $e$ form the Hamiltonian face then the arc $e$ is always the outermost arc. In such an arc diagram drawing the Hamiltonian face is the unbounded face. %Figure \ref{plain ham} loosely demonstrates the description.

%\begin{figure}[h]
%   \centering
%       \includegraphics[page=1,width=.45\textwidth]{cond1pic/Figure2.png}
% \caption{Planar embedding with a Hamiltonian face (up to equivalence)}
% \label{plain ham}
%\end{figure}

\end{remark}

%\begin{lemma}
%A graph $G$ has a planar embedding with a Hamiltonian circuit as a facial cycle if and only if it is isomorphic to a series-parallel graph built out of $\paral^e$ and $\bolt$ with no cut vertex and not isomorphic to $K_2$. Moreover, if one such embedding exists, then for any $s,t \in G$ which are consecutive on the Hamiltonian face there is an isomorphism $\phi$ to a series-parallel graph $(\wtilde G,s',t')$ built out of $\paral^e$ and $\bolt$ such that $\phi(s) = s'$ and $\phi(t) = t'$. 
%\end{lemma}

\begin{lemma}
Let $G$ be a graph that has a planar embedding with a Hamiltonian cycle as a facial cycle and let $s,t \in G$ be vertices consecutive on the Hamiltonian face. Then there is an isomorphism $\phi$ to a series-parallel graph $(\wtilde G,s',t')$ built out of $\paral^e$ and $\bolt$ such that $\phi(s) = s'$ and $\phi(t) = t'$. Conversely, if $(G,s,t)$ is a series-parallel built out of $\paral^e$ and $\bolt$ with no cut vertex then $G$ has a planar embedding with a Hamiltonian cycle as a facial cycle where $s,t$ are consecutive on the Hamiltonian face.
\end{lemma}

\begin{proof}

\begin{enumerate}

\item[$(\implies)$]

Let $\pi: G \ra \mbc$ be a planar embedding with a Hamiltonian face and let $s,t \in G$ be two consecutive vertices on the Hamiltonian face. Then as in Remark \ref{ham arc-diagram} we may view $G$ as a non-crossing arc diagram with $s$ the leftmost vertex, $t$ the rightmost vertex, and $e$ the outermost arc which completes the Hamiltonian face. 

Remove the arc $e$ from $G$. Then the biconnected components of the resulting graph are either single edges or are embedded in the plane as non-crossing arc diagrams with an arc connecting the leftmost and rightmost vertices (this is still true even when $e$ has a parallel edge in $G$). Let $H$ be one of these biconnected components. 

If $H$ is just a single edge then it is a series parallel graph. Otherwise, $H$ is already embedded as a non-crossing arc diagram such that all of the vertices of $H$ lie on a horizontal path with an arc connecting the leftmost vertex $s'$ and rightmost vertex $t'$. In particular, $(H,s',t')$ is a smaller graph with a marked pair of vertices satisfying the hypothesis, so by an inductive argument we see that $H$ is a series-parallel graph with source $s'$, terminal $t'$, and is built using only the prescribed operations. We now recover $G$ in two steps:
\begin{enumerate}[(i)]
\item
Taking the series join of all biconnected components.

\item
Taking the (restricted) parallel join with the single arc $e$

\end{enumerate}

In particular $G$ is series-parallel with source the leftmost vertex of the first biconnected component, which is $s$, and terminal the rightmost vertex of the last biconnected component, which is $t$.

\item[$(\impliedby)$]

Let $(G,s,t)$ be a series-parallel graph satisfying the criterion of the lemma. We prove the result by induction on $\min \set{\height{\Upsilon}}$, with $\Upsilon$ running over the decomposition trees for $(G,s,t)$. Fix $\Upsilon$ a decomposition tree of smallest height for $G$.   

If $\height \Upsilon = 0$ then $G = K_2$. If $\height \Upsilon = 1$ then since $G$ has no cut vertex $G = K_2\paral K_2$, for which the result is clear. Assuming the result for all $\height \Upsilon \leq k$ we prove the result for height $k+1$. Since $G$ has no cut vertex the root of $\Upsilon$ cannot be a $\bolt$ operation. Thus $G = H\paral e$, where
\[
H = H_1 \bolt H_2 \bolt \ldots \bolt H_n
\]
for some $n \geq 1$ and each $H_i$ is either $K_2$ or does not have a cut vertex. But each $H_i$ is built out of only $\set{\bolt, \paral^e}$ and has height strictly smaller than $k+1$, so by the induction hypothesis any $H_i$ has an arc diagram embedding as in Remark \ref{ham arc-diagram}. Concatenating the arc diagrams of the $H_i$ and adding $e$ as an outer arc above the rest shows that $G$ has a face which is a Hamiltonian cycle, namely, the outer face given by the horizontal path in the arc diagram and the arc $e$. In particular $s$ and $t$ are consecutive on this face.

%Hamilton path from the source to the terminal. Bolting these together gives a Hamilton path $\P$ from $s$ to $t$, which together with $e$ makes a Hamilton cycle. Define the total ordering $<$ on vertices of $G$ to be the ordering given by $\P$, with $s$ the minimal element. Choose a planar embedding $\pi \: G \ra \mbc$ such that $\P \ssq \mbr$ and $\pi$ respects the ordering. If $(a,b), (c,d) \in E(G)$ (with $a<b, c<d, a < c$) then we must have $d < b$ since $G$ cannot have a minor isomorphic to $K_4$.  Therefore, embedding the arcs $(a,b) \not \in \P$ by
%$$
%\pi'(\theta) := \frac{\pi(b)-\pi(a)}{2} e^{i\theta} + \frac{\pi(a)+\pi(b)}{2}, \midskip \theta \in [0,\pi]
%$$
%extends $\res{\pi}{\P}$ to a planar embedding of $G$ with a Hamiltonian face.
\end{enumerate}
\end{proof}

We recall the definition of the $\Upsilon$-dual from Definition~\ref{def upsilon} to state the next result, which gives a criterion for condition $T$ on series-parallel graphs.

\begin{cor} \label{T-for-co-Hamiltonian}
Let $G$ be a graph with a planar embedding with a Hamiltonian face and let $s,t \in G$ be consecutive on the Hamilton face. Invoking the previous lemma fix a decomposition tree $\Upsilon$ for $(G,s,t)$. If $V(G) = \set{s,t}$ then the $\Upsilon$-dual is a path. Otherwise, the $\Upsilon$-dual is a series-parallel graph satisfying condition $T$.
\end{cor}

\begin{proof}
Let $G^\vee$ be the $\Upsilon$-dual of $G$. By definition of the $\Upsilon$-dual we see that $G^\vee$ is a series-parallel graph built out of $\set{\paral, \bolt^e, \bolt_e}$. In particular, condition $T$ is stable under these operations by Lemma \ref{T-path-lift} and Lemma \ref{T-parallel-lift}. If $G^\vee$ is not a path then $\Upsilon^\vee$ contains at least one $\paral$ operation, so there must be one furthest from the root (under the obvious partial order). The subtree rooted at this $\paral$ builds a cycle, so in particular $T$ holds for this cycle. As previously mentioned, $T$ then extends to $G^\vee$. Otherwise, $\Upsilon$ contains only $\paral$ operations, so $V(G) = \set{s,t}$. 
\end{proof}

We are finally ready to produce a combinatorial condition on a graph which implies simultaneous combination and whence condition 1 for every edge.

%\begin{cor} \label{S-for-co-Hamiltonians}
%If $G$ is an $\Upsilon$-dual for a graph that has a planar embedding with a Hamiltonian face with two consecutive marked vertices and $G$ is not a cycle then $S(G \paral e)$ holds.
%\end{cor}

%\begin{proof}
%By the previous corollary we have that $G = e\paral H$. If $H$ is a series-parallel graph with at most one cycle then by Corollary \ref{cyclepaths} we obtain the result. Otherwise, $\height \Upsilon \geq 2$. We may assume by an inductive argument that $S(H)$ holds so all that is left to do is prove $S(G,e)$. Let $H = H_1 \paral \ldots \paral H_n$ with $n$ maximal. Either every $H_i$ is a path, in which case we are done, or at least one satisfies $H_i = H' \bolt \P$ with $\P$ a path and $H'$ an $\Upsilon'$-dual of a graph with a planar embedding with a Hamiltonian face. By the previous corollary we see that $T(H')$ and thus $T(H)$ holds. By Lemma \ref{bridge lemma} we get $S(G,e)$. 
%\end{proof}

\begin{cor} \label{S-for-co-Hamiltonians}
Let $H_i^\vee$ for $2 \leq i \leq n$ be graphs which either have a Hamiltonian face with two consecutive marked vertices, or are paths. View each of these as a series-parallel graph and let $G^\vee := H_1^\vee \bolt \ldots \bolt H_n^\vee$. Fix an arbitrary decomposition tree $\Upsilon^\vee$ for $G^\vee$ and let $G$ be the $\Upsilon^\vee$-dual. If $G^\vee$ has at least 4 vertices then $S(G,e)$ holds for every $e \in G$.   
\end{cor}

\begin{remark}
As in Remark \ref{ham arc-diagram} we may view each of the $H_i^\vee$ as non-crossing arc diagrams with a horizontal Hamilton path. Thus we may view $G^\vee$ as a non-crossing arc diagram with the Hamilton path from $s$ to $t$ as a horizontal line segment whose left endpoint is $s$ and right endpoint $t$ and all other edges as arcs above this line segment.
\end{remark}

\begin{proof}
By fixing $\Upsilon^\vee$ for $G^\vee$ we implicitly fix a decomposition tree $\Upsilon_i^\vee$ for each $H_i^\vee$. Indeed, $\Upsilon_i^\vee$ is the rooted subtree of $\Upsilon^\vee$ whose leaves are exactly those labelled with the edges of $H_i^\vee$. Note by definition of the $\Upsilon^\vee$-dual that $G = H_1 \paral \ldots \paral H_n$. If $H_i^\vee$ has at most 3 vertices then $\Upsilon_i^\vee$ has at most one $\bolt$ operation. Thus $\Upsilon_i$ has at most one $\paral$ operation so $H_i$ has at most one cycle. With this in mind we divide the proof into 3 cases: (1) each $H_i$ is a path; (2) each $H_i$ has at most 1 cycle; (3) $H_1^\vee$ has at least 4 vertices.  

\begin{enumerate}
\item
If each $H_i$ is a path then each $H_i^\vee$ has two vertices by Corollary \ref{T-for-co-Hamiltonian}. As $G^\vee$ has at least 4 vertices we have $n \geq 3$. In particular, $H_1 \paral H_2$ is a cycle so by Lemma \ref{cycle lemma} and Lemma \ref{bridge lemma} we have both $S(G)$ and $T(G)$.

\item
If $H_1$ has exactly one cycle then $H_1 = C \bolt \P$ with $C$ a cycle and $\P$ a path. Thus we have $T(H_1)$ by Corollary \ref{cyclepaths}. If $H_2$ is a path then by Lemma \ref{bridge lemma} we have $S(H_1 \paral H_2,e)$ for all $e \in H_2$. By Corollary \ref{cyclepaths} we have $S(H_1 \paral H_2, e)$ for all $e \in H_1$ and $T(H_1 \paral H_2)$. We apply the same techniques to the rest of the $H_i$ to get $S(G)$ and $T(G)$.

\item
As $H_1^\vee$ has arcs connecting its source and terminal we have that $H_1^\vee = \Gamma^\vee \paral e_1 \paral \ldots \paral e_k$. Since $H_1^\vee$ has at least 4 vertices so too does $\Gamma^\vee$, so it is a graph satisfying the hypothesis of Corollary \ref{S-for-co-Hamiltonians}. Thus from an inductive argument on $\height \Upsilon^\vee$ we may assume that $\Gamma$ satisfies both conditions $S$ and $T$. Since $H_1 = \Gamma \bolt e_1 \bolt \ldots \bolt e_k$ we have by Lemma \ref{T-path-lift} that $T(H_1)$ and by Lemma \ref{S-lift} that $S(H_1,e)$ for all $e \in \Gamma$. 

Now by symmetry and the previous cases either $H_2$ is a path or satisfies $T(H_2)$, so by Lemma \ref{bridge lemma} we have that $S(H_1 \paral H_2, e)$ for all $e \in \P$. Since we already have $S(H_1,e)$ for all $e \in \Gamma$ we have $S(H_1 \paral H_2,e)$ for all $e \in H_1$. By $T(H_1)$ we have $T(H_1 \paral H_2)$, so in particular if $H_2$ is a path or has exactly one cycle then we have $S(H_1 \paral H_2)$; on the other hand, if $H_2^\vee$ has at least 4 vertices then by symmetry we obtain $S(H_1 \paral H_2)$ anyway. We can repeat the symmetry argument to get $S(G)$ and $T(G)$.

\end{enumerate}
\end{proof}

In fact we can conclude substantially more. We can take \emph{any} graph and replace an edge with a piece as in Corollary~\ref{S-for-co-Hamiltonians} and conclude that condition 1 holds for edges from the piece.  The key is that Lemma \ref{S-lift} and Lemma \ref{cycle lemma} hold for any source-terminal graphs not just series-parallel graphs since the proofs never use more than that, so we can do a parallel join with any graph once we have a piece satisfying $S$. 

If $G$ is a graph, $(u,v) \in G$ is an edge, and $H$ is a source-terminal graph, then we can construct a new graph by deleting $(u,v)$ and gluing the source (resp. terminal) of $H$ to $u$ (resp. $v$). In effect we replace the edge $(u,v)$ with $H$.

\begin{prop} \label{replacement by co-Hamiltonian}
Let $H$ be a graph which is either a cycle or satisfies the hypothesis of Corollary \ref{S-for-co-Hamiltonians}. Let $G$ be any graph and replace any edge $(u,v) \in G$ which is not a self-loop with $H$; call the new graph $\wtilde G$. Then condition 1 holds for any $e \in H \ssq \wtilde G$ provided the loop number of $\wtilde G$ is at least 2. 
\end{prop}

Note that this proposition also gives an alternate (albeit considerably less elementary) proof for Proposition~\ref{e parellel} by replacing an edge by the cycle consisting of two parallel edges.

\begin{proof}
If $H$ is just a cycle, $G$ has at least one cycle, and $(u,v)$ is not a bridge, then by Lemma \ref{cycle lemma} we obtain simultaneous combination for $e$ and thus condition 1. Otherwise, we have $S(H,e)$ by Corollary \ref{S-for-co-Hamiltonians} so this implies $S(\wtilde G,e)$ and thus condition 1.
\end{proof}

To underscore the freedom available from the fact that $G$ is unrestricted, note that if we choose $G$ to be any series-parallel graph then we can construct a new series-parallel graph $\wtilde G$ such that $S(\wtilde G)$ holds by applying Proposition \ref{replacement by co-Hamiltonian} to every edge of $G$. We can also feed these graphs through Proposition \ref{replacement by co-Hamiltonian} into any other graph $\Gamma$ etcetera. 

% <added> Example for Proposition 6.16

	% The scaling factor of this figure is a bit different from the rest of the figures in the paper due to the size.
\begin{figure}[h]
   \centering
       \includegraphics[page=1,width=.7\textwidth]{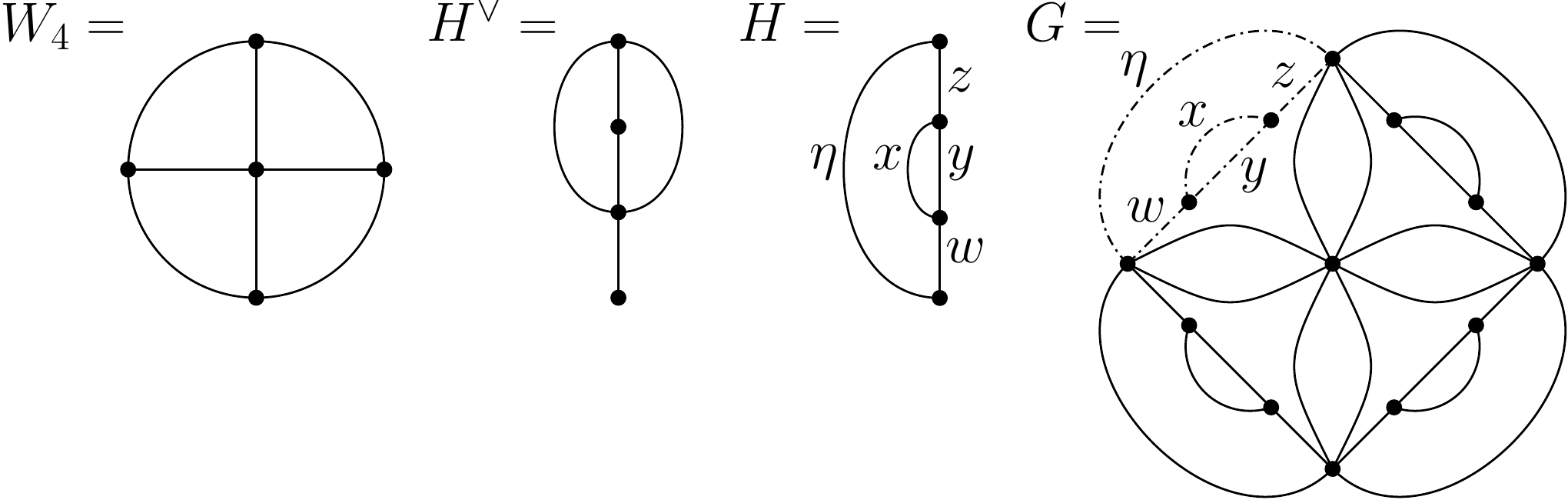}
 \caption{Example for Proposition \ref{replacement by co-Hamiltonian}}
 \label{fig:replacement example}
\end{figure}

We provide an example of this procedure using the graphs in Figure \ref{fig:replacement example}. Recall that the wheel graph $W_4$ satisfies condition 1 for spoke edges and does not satisfy condition 1 for rim edges. The graph $H^\vee$ is the series join of graphs satisfying the conditions of Corollary \ref{S-for-co-Hamiltonians}. We see that $H$ is an $\Upsilon$-dual of $H^\vee$ and has at least $4$ vertices, so in particular we have $S(H,e)$ for all $e \in H$ by Corollary~\ref{S-for-co-Hamiltonians}. 

Replacing the rim edges of $W_4$ with $H$ and doubling the spoke edges gives the graph $G$. By Proposition \ref{replacement by co-Hamiltonian} we see that $S(G,e)$ holds if $e$ lies in a copy of $H$. Additionally, if $e$ is not in a copy of $H$ then $e$ is a double edge. Therefore, condition 1 holds for every edge of $G$. 

% Start of explicit equations in the example.

We can verify that condition 1 holds explicitly in this example for the edge $\eta \in G$.  Calculations were done with Magma \cite{MAGMA} and a script is provided at \cite{KMYscript} and also included in the arXiv source. For $H$ we have
\begin{align*}
	\kirk{H} &= (x+y)\eta + (xy + (x+y)(z+w)) \\
	\breaker{H} &= \eta(xy + (x+y)(z+w)).      
\end{align*}
Let $n := \deg(\breaker{H}) = 3$ and let 
\begin{align*}
    A_x &:= \frac{1}{2}yz + \frac{1}{2}yw - \frac{1}{2}y\eta, \\
    A_y &:= xy + xz + xw + x\eta + \frac{1}{2}yz + \frac{1}{2}yw + \frac{3}{2}y\eta, \\
    A_z &:= 0, \\
    A_w &:= -xy - xz - xw - x\eta - \frac{3}{2}yz - \frac{3}{2}yw + \frac{1}{2}y\eta - z^2 - 2zw - w^2, \\
    A_\eta &:= 0, \\
    C   &:= y.
\end{align*}
One can check that 
	\begin{align*}
	  \kirk{H} &= \sum_{j \in H} A_j \dkirk{H}{\eta j} \\
	  \breaker{H} &= \sum_{j \in H} A_j \dbreaker{H}{\eta j} + C\dbreaker{H}{\eta}. 
	\end{align*}
% I should redraw the graph of $G$ so that it distinguishes a copy of $H$ by colour. The edges of this copy should also be labelled.
That is, we have satisfied $S(H,\eta)$. Let $\Gamma$ be the undotted subgraph of $G$. Note that $\Gamma$ has $11$ vertices and $23$ edges, so each spanning tree excludes $13$ edges. That is, $\deg(\kirk{\Gamma}) = 13$ and $\deg(\breaker{\Gamma}) = 14 =: m$. We note $G = H \paral \Gamma$, so
	\begin{align*}
		\kirk{G} &= \bra{(x+y)\eta + (xy + (x+y)(z+w))} \breaker{\Gamma} + \eta(xy + (x+y)(z+w)) \kirk{\Gamma} \\
		\breaker{G} &= \eta(xy + (x+y)(z+w))\breaker{\Gamma}.      
	\end{align*}
The explicit forms of $\kirk{\Gamma}$ and $\breaker{\Gamma}$ are too large to be included in print. However, for the computations checking $S(G,\eta)$ it is enough to know that they are homogeneous of degrees $m-1$ and $m$ respectively. Setting
	\begin{align*}
		B_j := \piecewise{A_j + \frac{C}{n+m-2}\bra{mx_j}}{j \in \{x,y,z,w,\eta\}}{-\frac{(n-2)C}{n+m-2}x_j} 
	\end{align*}
	\begin{align*}
		C' := \bra{1-\frac{m}{n+m-2}}C
	\end{align*}
	we obtain polynomials which verify $S(G,\eta)$. In particular, condition 1 is satisfied for $\eta \in G$. 

% Conclude condition1 holds. 

%\begin{remark}[Editorial note]
%Perhaps, since we mostly care about the duals of graphs with planar Hamiltonian embeddings, that we should call such graphs \emph{co-Hamiltonian}.
%\end{remark}

\section{Conclusion}\label{sec conclusion}

Our results significantly increase the graphs and edges for which we know whether or not condition 1 holds.  This means that there are many more graphs for which the tools of \cite{Acond} can be applied.

After our investigations of multiple edges, wheel graphs, and series parallel graphs, we are left with some questions.  Lemma~\ref{lemma s} was inspired by the possibility that condition 1 might carry through $\Delta$ to $Y$ transformations.  Consider graphs containing the following structure.

\begin{center}
\includegraphics[width=0.44\textwidth]{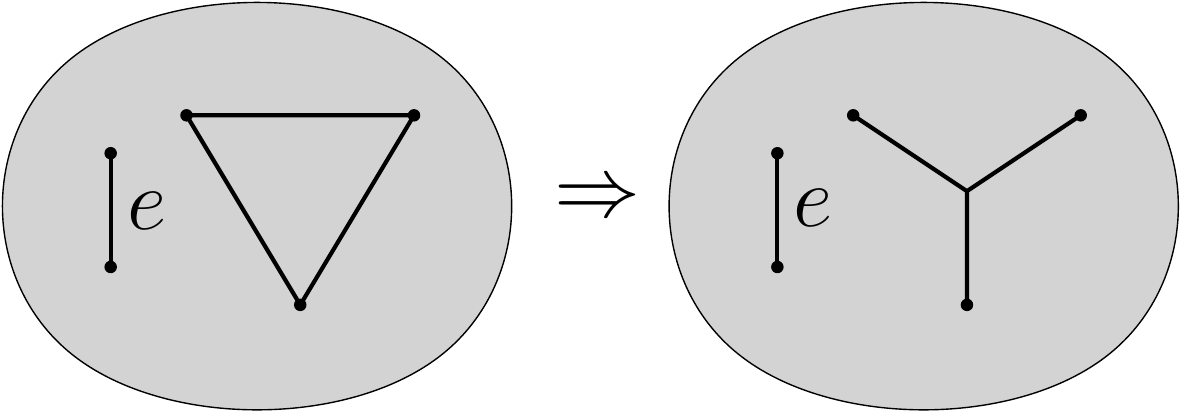}
\end{center}

Suppose we have a graph $ G_{\Delta} $, as on the left in the above diagram, with 3 edges forming a $ \Delta $ shape and a distinct fourth edge $ e $.  The assertion that $\Delta$ implies $Y$ for condition 1 would mean that if $ 1(G_{\Delta}, e) $ is true then we could replace the $ \Delta $-forming edges with edges forming a $ Y $, to make the graph $ G_Y $, and $ 1(G_{Y}, e) $ would still be true.  We suspect that this is the case.

The following heuristic has been very useful in this paper. The equations for condition 1 are unstable under natural operations such as gluing two graphs along $n>1$ vertices or adding parallel edges. However,  properties or transformations one suspects to imply condition 1 sometimes are preserved by these operations. For example, when we investigated whether or not condition 1 was true for all regular edges in series-parallel graphs, we found extra conditions that were well behaved under the operation of gluing two series-parallel graphs together which could be related to condition 1.  A study of these conditions then led us to the example of Figure~\ref{counterexample} which was a counterexample to our initial hopes.  The relationship between condition 1 and the $\Delta-Y$ transformation could be approached using this idea.  Furthermore, Aluffi observes \cite[page 5]{Acond} that in general ``condition 1 depends on the global features of the graph''; we can see the simultaneous combination conditions as serving to correct this enabling us to obtain local results like Proposition~\ref{replacement by co-Hamiltonian} even for condition 1 itself.

One could hope for a full characterization of graph edge pairs satisfying condition 1.  We have been looking for a structural graph theoretic characterization, but one could also ask about the computational question -- what is the computational complexity of checking condition 1 on a graph edge pair?  
Continuing our series-parallel investigations, both the structural and computational questions could be asked for specific classes of graphs.
%Although our discussion of series-parallel graphs is incomplete because we do not give a full characterization of which edges do and do not satisfy condition 1, there may be other nontrivial classes of graphs that more naturally lend themselves to such a characterization.  
It would also be interesting to study the proportion of edges that satisfy condition 1 for large graphs.  We suspect that this proportion may asymptotically approach zero, simply because as graphs become larger the condition becomes more complicated and difficult to satisfy, however we would like to have a more rigorous analysis of this problem.  

Finally one could consider Aluffi's condition 2.  Condition 2 seems to be much more geometric and less graph theoretic, so we expect it to be less amenable to this sort of analysis.

\bibliographystyle{plain}
\bibliography{main}

\end{document}